\documentclass[12pt,a4paper]{article}

\usepackage{graphicx}
\usepackage{amsmath}
\usepackage{amsfonts}
\usepackage{amssymb}
\usepackage{amsthm}
\usepackage{floatflt}
\usepackage{setspace}
\usepackage{pdfpages}
\usepackage{fancyhdr}

\usepackage{makeidx}
\makeindex

\theoremstyle{plain}
\newtheorem{thm}{Theorem}[section]
\newtheorem{hopefully-a-thm}{Wish}[section]
\newtheorem{cor}[thm]{Corollary}
\newtheorem{prop}[thm]{Proposition}
\newtheorem{lemma}[thm]{Lemma}

\theoremstyle{definition}
\newtheorem{definition}{Definition}[section]
\theoremstyle{remark}
\newtheorem*{remark}{Remark}

\newtheorem*{example}{Example}

\newcommand{\N}{\mathbf{N}}
\newcommand{\R}{\mathbf{R}}

\newcommand{\C}{\mathbf{C}}
\newcommand{\Q}{\mathbf{Q}}
\newcommand{\D}{\mathbf{D}}

\newcommand{\T}{\mathbf{T}}
\newcommand{\tr}{\mathrm{tr}\;}
\newcommand{\Hpl}{\mathbf{H}}
\renewcommand{\H}{\mathcal{H}}

\newcommand{\mint}{\overset{\curvearrowright}{\int}}

\newcommand{\prodr}{\overset{\curvearrowright}{\prod}}
\newcommand{\prodl}{\overset{\curvearrowleft}{\prod}}

\newcommand{\diag}{\mathrm{diag}}

\renewcommand{\Re}{\mathrm{Re}\,}
\renewcommand{\Im}{\mathrm{Im}\,}

\newcommand{\defin}[1]{{\em #1}}

\newcommand{\adj}{{\mbox{*}}}

\newcommand{\tinypf}[1]{}

\newcommand{\sigmamax}{\sigma_{\text{max}}}

\newcommand{\Sch}{\mathcal{S}}
\newcommand{\Img}{\mathrm{Im}\,}

\newcommand{\BV}{\mathrm{BV}}
\newcommand{\var}{\mathrm{var}}
\newcommand{\osc}{\mathrm{osc}}

\newcommand{\ch}{\mathbf{1}}

\newif\ifdetail

\newcommand{\todo}[1]{\ifdetail TODO{\ifx&#1&{: #1}\fi}\fi}

\detailtrue

\numberwithin{equation}{section}

\makeatletter

\newcommand*{\betreuer}[1]{\def\@betreuer{#1}}
\betreuer{}
\newcommand*{\ausarbeitungstyp}[1]{\def\@ausarbeitungstyp{#1}}
\ausarbeitungstyp{}
\newcommand*{\institut}[1]{\def\@institut{#1}}
\institut{}
\newcommand*{\authornew}[1]{\def\@authornew{#1}}
\authornew{}

\renewcommand\maketitle{\begin{titlepage}
  \let\footnotesize\small
  \let\footnoterule\relax
  \let \footnote \thanks
  \null\vfil

\begin{center}%

\parbox{12cm}{\centering\LARGE\bfseries \@title \par}\\
{\large 
\vspace{.917em}
\@authornew}\\
\vspace{.917em}
\vspace{.917em}
{\normalsize 
 \@date}
\vspace{13.75em}

{\normalsize \@ausarbeitungstyp}\\
\vspace{.917em}
{\normalsize \@betreuer}\\
\vspace{.917em}
\centerline{{\normalsize\sc \@institut}}
\vspace{13.75em}

\centerline{{\normalsize\sc Mathematisch-Naturwissenschaftliche Fakult\"at der}}
\vspace{.917em}
\centerline{{\normalsize\sc Rheinischen Friedrich-Wilhelms-Universit\"at Bonn}}

\end{center}

\clearpage
\thispagestyle{empty}\mbox{}
\clearpage
\pagenumbering{arabic}

\end{titlepage}

  \setcounter{footnote}{0}%
  \global\let\maketitle\relax
  \global\let\@author\@empty
  \global\let\@date\@empty
  \global\let\@title\@empty
  \global\let\title\relax
  \global\let\author\relax
  \global\let\date\relax
  \global\let\and\relax
}
\makeatother

\authornew{Joris Roos}
\date{4th February 2014}
\betreuer{Advisor: Prof. Dr. Christoph Thiele}
\institut{Mathematical Institute}
\title{Inner-outer factorization of analytic matrix-valued functions}
\ausarbeitungstyp{Master's Thesis  Mathematics}

\begin{document}
\maketitle

\pagestyle{fancy}
\setcounter{page}{3}

\tableofcontents
\newpage

\section{Introduction}

It is well known \cite{Garnett}, \cite{Rudin2} that an analytic function $f\not\equiv 0$ on the unit disk $\D$ in the complex plane, satisfying $|f(z)|\le 1$ for all $z\in\D$, can be uniquely factored as
\begin{align}
\label{eqn:scinnerouter}
f=m_f\cdot q_f
\end{align}
where $m_f$ is an \emph{inner function}\index{scalar inner function}, i.e. a bounded analytic function satisfying $|m_f(z)|=1$ almost everywhere on $\T=\partial\D$ and $q_f$ is an \emph{outer function}\index{scalar outer function}, i.e. $-\log|q_f|$ is the Poisson extension of an absolutely continuous positive measure. Inner functions can be further factored into a Blaschke product\index{Blaschke product}, that is, a convergent product of functions of the form\footnote{With obvious modifications if $z_0=0$.} $b_n(z)=\frac{|z_0|}{z_0}\frac{z_0-z}{1-z\overline{z_0}}$, $z_0\in\D$, and a non-vanishing inner function.

This work describes that factorization for the case that $f$ is a bounded analytic matrix-valued function on the unit disk (we will abbreviate the term \emph{matrix-valued function} by \emph{mvf}\index{mvf} from now on).
It should be noted that a factorization of the type \eqref{eqn:scinnerouter} can be achieved also for a more general class of functions taking as values operators on a Hilbert space. This theory was developed by Sz.-Nagy et al. \cite{Sz.-Nagy} and uses abstract machinery of operator theory.

However, our goal is to provide a more explicit understanding for the case of matrices. An appropriate generalization of the Poisson integral representations for scalar inner and outer functions turns out to be given by multiplicative integrals. These are defined in a similar manner to classical Riemann-Stieltjes integrals, where the usual Riemann sums are replaced by products and we are integrating over matrices instead of complex numbers. A motivation for considering multiplicative integrals originates in the study of the nonlinear Fourier transform or scattering transform \cite{TaoThieleTsai}, a certain discrete nonlinear analogue of the classical Fourier transform. Multiplicative integrals also arise naturally in the theory of canonical systems of ordinary differential equations, as they can be interpreted as monodromy matrices of such systems \cite{Arov-DymI}, \cite{Arov-DymII}. In physics, particularly in quantum field theory, multiplicative integrals play a central role and are known as time-ordered or path-ordered exponentials.

Multiplicative representations for certain classes of analytic mvfs have been developed by V.P. Potapov \cite{Potapov}. The inner-outer factorization of bounded analytic mvfs was found by Ginzburg \cite{Ginzburg}, but he omits the details of his proofs. We are trying to fill in the gaps.

Let us proceed to describe the main result. This requires a few definitions, which we will briefly state now. They will be repeated and covered in greater detail in later sections.

A bounded analytic mvf is a mvf  whose entries are bounded analytic functions. A Blaschke-Potapov product is a possibly infinite product of mvfs on the unit disk of the form
\begin{align}
b(z)=I-P+\frac{|z_0|}{z_0}\frac{z_0-z}{1-z\overline{z_0}}P
\end{align}
for some $z_0\in\D$ and an orthogonal projection $P$. By convention, we also allow a Blaschke-Potapov product to be multiplied by a unitary constant. We call a mvf $A$ on an interval $[a,b]$ increasing if it is Hermitian and $A(t)-A(s)$ is positive semidefinite for all $t,s\in[a,b]$ with $t\ge s$. By an outer mvf we mean a bounded analytic mvf on the unit disk of the form
\begin{align}
\label{eqn:introouter}
E(z)=U\mint_0^{2\pi} \exp\left(\frac{z+e^{i\varphi}}{z-e^{i\varphi}} M(\varphi)d\varphi\right)
\end{align}
where $U$ is a unitary constant, $M$ an integrable and Hermitian mvf, whose least eigenvalue is bounded from below. The symbol $\mint$ denotes a multiplicative integral, which we will discuss in detail in Section \ref{sect:multint}.
A pp-inner inner function is a mvf on the unit disk taking the form
\begin{align}
\label{eqn:introppinner}
S_{pp}(z)=U\prod_{k=1}^m \mint_0^{l_k} \exp\left(\frac{z+e^{i\theta_k}}{z-e^{i\theta_k}}\,dE_k(t)\right)
\end{align}
for a unitary constant $U$, $m\in\N\cup\{\infty\}$, $l_k>0$, $\theta_k\in[0,2\pi)$ and increasing mvfs $E_k$ with $\tr E_k(t)=t$. By an sc-inner function we mean a mvf on the unit disk which can be written as
\begin{align}
\label{eqn:introscinner}
S_{sc}(z)=U\mint_0^{2\pi} \exp\left(\frac{z+e^{i\varphi}}{z-e^{i\varphi}}dS(\varphi)\right)
\end{align}
for a unitary constant $U$ and a singular continuous increasing mvf $S$. More details and equivalent characterizations for these definitions will be given in Section \ref{sect:innerouterfact}. We can now state the main theorem.

\begin{thm}
\label{thm:main}
Let $A$ be a bounded analytic function on $D$ such that $\det A~\not\equiv~0$. Then there is a Blaschke-Potapov product $B$,  a pp-inner mvf $S_{pp}$, an sc-inner mvf $S_{sc}$ and an outer mvf $E$ such that
\begin{align}
A(z)=B(z)S_{pp}(z)S_{sc}(z)E(z)
\end{align}
for all $z\in\D$. Moreover, this factorization is unique in the sense that the factors are uniquely determined up to multiplication with a unitary constant.

Also, the function $M$ in the representation \eqref{eqn:introouter} is uniquely determined up to changes on a set of measure zero.
\end{thm}

It is a natural question, whether also the functions $S, E_k$ in \eqref{eqn:introppinner}, \eqref{eqn:introscinner} are uniquely determined. We will see that the answer to that question is, at least in the case of the functions $E_k$, negative.

This thesis is structured as follows. In Section \ref{sect:multint} we develop the theory of multiplicative Riemann-Stieltjes integrals to an extent sufficient for our purpose. Section \ref{sect:contractivemvfs} presents Potapov's fundamental theorem on multplicative representations of contractive mvfs (Theorem \ref{thm:potapov}). We also discuss convergence and uniqueness questions on Blaschke-Potapov products. The next section is devoted to the proofs of Theorem \ref{thm:main} and some additional properties of inner and outer mvfs. In the appendix we have assembled several basic facts which are needed in the text. This includes in particular a proof of Herglotz' representation theorem for mvfs, which is used in a crucial step in the proof of Potapov's fundamental theorem.\\

{\bf Acknowledgement.} I would like to express my gratitude to my advisor Prof. Christoph Thiele, who has supported me throughout the work on this thesis and provided me with helpful comments.

\section{Multiplicative integrals}
\label{sect:multint}

The multiplicative integrals which we are concerned with are certain multiplicative analogues of classical Riemann-Stieltjes integrals.

Multiplicative integrals first originated in the work of V. Volterra who considered them in 1887 for the purpose of studying systems of ordinary differential equations. L. Schlesinger later formulated Volterra's concepts in a rigorous framework, cf. \cite{Schlesinger1}. An overview of the subject is given in \cite{Slavik}. However, the focus there is on multiplicative Riemann and Lebesgue integrals. Multiplicative integrals of Stieltjes type are discussed in \cite[Appendix \S 1]{Potapov} and \cite[\S 25]{Brodskii}.

\subsection{Definition}
\label{multint}

Let us first fix some notation and conventions which will be used not only in this section, but throughout the entire text. 

The space of $n\times n$ matrices with entries in $\C$ will be denoted by $M_n$. We equip $M_n$ with the matrix norm\index{matrix norm} $\|A\|=\sup_{\|x\|_2=1} \|Ax\|_2$ where $\|\cdot\|_2$ denotes the Euclidean norm in $\C^n$. Several properties and estimates for this norm are given in Appendix \ref{subsect:matrixnorm}. They will be used without further reference throughout the text.

We call a matrix $A\in M_n$ \defin{positive}\index{positive matrix} and write $A\ge 0$, if it is Hermitian and positive semidefinite. For a positive definite Hermitian matrix $A$ we write $A>0$ and call it \defin{strictly positive}\index{strictly positive matrix}.

A mvf $A:[a,b]\rightarrow M_n$ is called \defin{increasing}\index{increasing} if it is Hermitian and monotonously increasing, i.e. $A(t)\ge A(s)$ whenever $t\ge s$. Likewise, $A$ is called \defin{strictly increasing}\index{strictly increasing} if $A(t)>A(s)$ whenever $t>s$. The terms \defin{decreasing}\index{decreasing} and \defin{strictly decreasing}{\index{strictly decreasing} are defined accordingly.

\begin{definition}
Let $a\le b$. A \defin{subdivision}\index{subdivision} or \defin{partition}\index{partition} $\tau$ of the interval $[a,b]$ is a finite set of real numbers $\{t_i\in[a,b]\,:\,i=0,\dots,m\}$ such that
$$a=t_0\le t_1\le\cdots\le t_m=b$$
Define $\Delta_i \tau=t_i-t_{i-1}$ for $i=1,\dots,m$ and $\nu(\mathcal{T})=\max_i \Delta_i \tau$. Moreover, given a mvf $E:[a,b]\rightarrow M_n$, we set $\Delta_i E=\Delta_i^\tau E=E(t_i)-E(t_{i-1})$ for $i=1,\dots,m$. Also define
$$\var_{[a,b]}^\tau E = \sum_{i=1}^m \|\Delta_i E\|$$
Then $E$ is called of \defin{bounded variation}\index{bounded variation}\index{$\var_{[a,b]} E$} if 
$$\var_{[a,b]} E=\sup_{\tau\in\mathcal{T}^b_a} \var^\tau_{[a,b]} E<\infty$$
The space of bounded variation functions (\defin{BV-functions})\index{$\mathrm{BV}([a,b]; M_n)$} with values in $M_n$ is denoted by $\BV([a,b]; M_n)$. 
If $n=1$, we write $\BV([a,b])$. For $f\in\BV([a,b]; M_n)$, we call
$$|E|(t)=\var_{[a,t]} E$$
the \defin{total variation function}\index{total variation}\index{$|E|(t)$} of $E$. It should not be confused with $\|E(t)\|$, which is the matrix norm of $E(t)$.\\
Given a subdivision $\tau$, choose intermediate points $\xi=(\xi_i)_{i=1,\dots,m}$ with $\xi_i\in[t_{i-1},t_i]$. By $\mathcal{T}^b_a$\index{$\mathcal{T}^b_a$} we denote the set of tagged partitions \index{tagged partition} $(\tau,\xi)$ such that $\tau$ is a subdivision of the interval $[a,b]$ and $\xi$ is a choice of corresponding intermediate points.
Given also a function $f$ on $[a,b]$ with values in $\C$ or $M_n$ we define\index{$P(f,E,\tau,\xi)$}
$$P(f,E,\tau,\xi)=P(\tau,\xi)=\prodr^m_{i=1} \exp(f(\xi_i)\Delta_i E)$$
Here $\prodr_{i=1}^m A_i=A_1\cdot A_2\cdots A_m$ denotes multiplication of the matrices $(A_i)_i$ from left to right. We will also often simply write $\prod_{i=1}^m A_i$ for $\prodr_{i=1}^m A_i$. Similarly, $\prodl_{i=1}^m A_i~=~A_m\cdot~A_{m-1}~\cdots~A_1$ denotes multiplication from right to left.
\end{definition}

The $P(\tau,\xi)$ form a net with respect to the directed set $(\mathcal{T}^b_a,\le)$, where we say that $(\tau,\xi)\le (\tau^\prime, \xi^\prime)$ if $\tau^\prime\subset\tau$. Note that $\tau^\prime\subset\tau$ implies $\nu(\tau)\le\nu(\tau^\prime)$, but the converse is not true.

\begin{definition}Let $X=\C$ or $X=M_n$. For a matrix $P\in M_n$ and functions $f~:~[a,b]\rightarrow X$, $E:[a,b]\rightarrow M_n$ we say that $P$ is the \defin{(right) multiplicative Stieltjes integral}\index{multiplicative Stieltjes integral} corresponding to this data if the net $(P(\tau,\xi))_{(\tau,\xi)\in\mathcal{T}^b_a}$ converges to $P$:
\begin{align}
\label{eqn:netlimit}
P=\lim_{\nu(\tau)\rightarrow 0} P(\tau,\xi)
\end{align}
i.e. for every $\epsilon>0$ there exists a $(\tau_0,\xi_0)\in\mathcal{T}^b_a$ such that
$$\|P(\tau,\xi)-P\|<\epsilon\;\text{for every}\;(\tau,\xi)\le (\tau_0,\xi_0)$$
In words, we will often refer to this as "the limit as $\nu(\tau)\rightarrow 0$".
We introduce the notation
$$P=\mint_a^b \exp(f\,dE)=\mint_a^b e^{f\,dE}$$
For short we also write $f\in\mathcal{M}^b_a[E]$ to denote the existence of $\mint^b_a\exp(f\,dE).$
\end{definition}

For the remainder of this section we will be concerned with criteria for the existence of multiplicative integrals.
\begin{lemma}[Cauchy criterion]\index{Cauchy criterion}
\label{lemma:cauchycrit}
Suppose that $f,E,a,b$ are as above and that for every $\epsilon>0$ there exists a $(\tau_0,\xi_0)\in\mathcal{T}^b_a$ such that for all $(\tau,\xi),(\tau^\prime,\xi^\prime)\le (\tau_0,\xi_0)$ we have
$$\|P(f,E,\tau,\xi)-P(f,E,\tau^\prime,\xi^\prime)\|<\epsilon$$
Then the integral $\mint^b_a \exp(f\,dE)$ exists. The converse also holds.
\end{lemma}

\begin{proof}
The condition in the lemma means that $(P(\tau,\xi))_{(\tau,\xi)}$ is a Cauchy net on the complete space $M_n$. Hence it converges.
\end{proof}

\begin{prop}
\label{prop:mintexist1}
Let $f:[a,b]\rightarrow\C$ be Riemann integrable and $E:[a,b]\rightarrow M_n$ Lipschitz continuous. Then the multiplicative integral $\mint^b_a\exp(f\,dE)$ exists.
\end{prop}

\begin{proof}
By Lebesgue's criterion\index{Lebesgue's criterion} for Riemann integrability we can choose $M~>~0$ such that $|f(t)|\le M$ for all $t\in[a,b]$. Let $L>0$ be a Lipschitz constant for $E$ and set $C=M\cdot L$. Now consider $(\tau,\xi),(\tau^\prime,\xi^\prime)\in\mathcal{T}^b_a$ and assume that $\tau^\prime$ coincides with $\tau$ except on the subinterval $[t_{k-1},t_k]$ for some fixed $k=1,\dots,m$, where it is given by
$$t_{k-1}=s_0\le s_1\le \cdots \le s_l=t_k$$
We denote the intermediate points $\xi^\prime$ in $[t_{i-1},t_i]$ by $\zeta_j\in [s_{j-1},s_j]$ for $j=1,\dots,l$.
Then
\begin{align}
\label{PPest1}
&\begin{aligned}
&\|P(\tau,\xi)-P(\tau^\prime,\xi^\prime)\|\\
&\qquad \le e^{\sum_{j\not=k}|f(\xi_j)|\cdot\|\Delta_j E\|}\|\exp(f(\xi_k)\Delta_k E)-\prod_{j=1}^l \exp(f(\zeta_j)(E(s_{j-1})-E(s_j)))\|\\
&\qquad \le e^{C(b-a)}\|\exp(f(\xi_k)\Delta_k E)-\prod_{j=1}^l \exp(f(\zeta_j)(E(s_{j-1})-E(s_j)))\|
\end{aligned}
\end{align}
Set $A_j=f(\zeta_j)(E(s_j)-E(s_{j-1}))$. Note that
\begin{align}
\label{sumAjest}
\sum_{j=1}^l \|A_j\|\le C\sum_{j=1}^l (s_j-s_{j-1})=C\Delta_k\tau
\end{align}
We now use the power series expansion of $\exp$ to see
\begin{align}
\label{exp2exp}
\prod^l_{j=1}\exp(A_j)=I+\sum^l_{j=1} A_j+R
\end{align}
Apply \eqref{sumAjest} to estimate the remainder term $R$ as follows
\begin{align}
\|R\|\le \exp\left(\sum^l_{j=1} \|A_j\|\right)-1-\sum_{j=1}^l \|A_j\|
\le C^2 (\Delta_k\tau)^2 e^{C(b-a)}
\end{align}
Similarly we obtain
\begin{align}
\label{exp1exp}
\exp(f(\xi_k)\Delta_k E)=I+f(\xi_k)\Delta_k E+R^\prime
\end{align}
where $R^\prime$ also satisfies $\|R^\prime\|\le C^2(\Delta_k\tau)^2 e^{C(b-a)}$. Plugging \eqref{exp1exp} and \eqref{exp2exp} into \eqref{PPest1} yields
\begin{align}
\label{PPest2}
\begin{aligned}
\|P(\tau,\xi)-P(\tau^\prime,\xi^\prime)\| &\le Le^{C(b-a)}\sum_{j=1}^l |f(\xi_k)-f(\zeta_j)|(s_j-s_{j-1}) + 2C^2(\Delta_k\tau)^2e^{2C(b-a)}\\
&\le Le^{C(b-a)}\osc_{[t_{k-1},t_k]} f\cdot \Delta_k \tau+2C^2(\Delta_k\tau)^2e^{2C(b-a)}
\end{aligned}
\end{align}
where for an interval $I\subseteq[a,b]$, $\osc_I f=\sup\{|f(t)-f(s)|\,:\,s,t\in I\}=\sup_I f - \inf_I f$ denotes the oscillation\index{oscillation} of $f$ on $I$.\\
Now let $(\tau^\prime,\xi^\prime)\in\mathcal{T}^b_a$ be arbitrary and write $\tau^\prime=\{s^{(k)}_j\,:\,k=1,\dots,m,\,j=0,\dots,l_k\}$ where
$$a=t_0=s^{(1)}_0\le s^{(1)}_1\le \cdots \le s^{(1)}_{l_1}=t_1=s^{(2)}_0\le\cdots\le s^{(m)}_{l_m}=t_m=b$$
Then we apply the above estimate \eqref{PPest2} $m$ times to obtain
\begin{align}
\label{PPest3}
\|P(\tau,\xi)-P(\tau^\prime,\xi^\prime)\| \le Le^{C(b-a)}\sum_{k=1}^m \osc_{[t_{k-1},t_k]} f\cdot \Delta_k \tau+2C^2e^{2C(b-a)}(b-a)\nu(\tau)
\end{align}
Since $\osc_I f=\sup_I f-\inf_I f$, the Darboux definition of Riemann integrability by upper and lower sums implies that we can make the sum on the right hand side arbitrarily small for small enough $\nu(\tau)$. The claim now follows from Lemma \ref{lemma:cauchycrit}.
\end{proof}

\begin{prop}
\label{prop:mintexist2}
Let $f:[a,b]\rightarrow\C$ be continuous and $E:[a,b]\rightarrow M_n$ be a mvf of bounded variation. Then the multiplicative integral $\mint^b_a\exp(f\,dE)$ exists.
\end{prop}
The proof is very similar as for the last proposition (see \cite[Appendix \S 1.1]{Potapov}, so we omit it. We do not claim that these existence results are in any way optimal. However, they are sufficient for the purpose of this thesis.

\subsection{Properties}
In this section we state and prove several important properties for multiplicative integrals. Among them are in particular a formula for the determinant, a change of variables formula and some estimates relating multiplicative integrals to (additive) Riemann-Stieltjes integrals.\\

First let us consider the case when the multiplicative integral reduces to an additive integral.

\begin{prop}
If $f\in\mathcal{M}_a^b[E]$ and the family of matrices $\{E(t)\,:\,t~\in~[a,b]\}$ commutes, then
$$\mint_a^b \exp(f(t)dE(t))=\exp\left(\int_a^b f(t) dE(t)\right)$$
In particular, this is always the case if $n=1$.
\end{prop}

\begin{proof}
This follows from the relation $e^{A+B}=e^A e^B$ which holds if $AB~=~BA$.
\end{proof}

The next property will be of fundamental importance in later arguments and allows us to decompose multiplicative integrals into products with respect to a decomposition of the interval. 
\begin{prop}
\label{prop:mintsep}
Let $a\le b\le c$. If $f\in\mathcal{M}^c_a[E]$, then
$f\in\mathcal{M}^b_a[E]\cap\mathcal{M}^c_b[E]$ and
$$\mint_a^c \exp(f\,dE)=\mint_a^b \exp(f\,dE) \mint_b^c \exp(f\,dE)$$
\end{prop}

\begin{proof}
To a partition of $[a,b]$ or $[b,c]$ we can always add one more point to make it a partition of $[a,c]$. Therefore Lemma \ref{lemma:cauchycrit} implies $f\in\mathcal{M}^b_a[E]~\cap~\mathcal{M}^c_b[E]$.\\
For brevity let us write $I_a^c, I^b_a, I^c_b$, respectively for the three multiplicative integrals in the claim. On the other hand, given $(\tau,\xi)\in\mathcal{T}^c_a$, we find $(\tau_0,\xi_0)\in\mathcal{T}^b_a$, $(\tau_1,\xi_1)\in\mathcal{T}^c_b$ such that $\nu(\tau_i)\le \nu(\tau)$ for $i=0,1$ and $P(\tau,\xi)=P(\tau_0,\xi_0)P(\tau_1,\xi_1)$. Let us choose $\epsilon>0, \delta>0$ such that for $(\tau,\xi)\in\mathcal{T}^c_a$ with $\nu(\tau)<\delta$ we have $\|P(\tau,\xi)-I_a^c\|<\epsilon$ and the same holds correspondingly for such partitions of the subintervals $[a,b]$, $[b,c]$. Then,
$$\|I_a^c - I_a^b I_b^c\|\le \|I_a^c-P(\tau,\xi)\|+\|I_a^bI_b^c-P(\tau,\xi)|<\epsilon+\|I_a^bI_b^c-P(\tau_0,\xi_0)P(\tau_1,\xi_1)\|$$
Now applying the identity $xy-zw=\frac{1}{2}\left((x-z)(y+w)+(y-w)(x+z)\right)$, we can estimate the remaining term as
$\|I_a^bI_b^c-P(\tau_0,\xi_0)P(\tau_1,\xi_1)\|\le C\epsilon$, where $C>0$ is an appropriate constant. Letting $\epsilon\rightarrow 0$ we obtain $I_a^c=I_a^bI_b^c$.
\end{proof}

\begin{prop}
\label{prop:mintconst}
Let $A$ be a constant matrix. Then 
$$\mint_a^b \exp(A\,d(tI)) = \exp((b-a)A)$$
\end{prop}

\begin{proof}
Let $(\tau,\xi)\in\mathcal{T}^b_a$. Then
$$P(\tau,\xi)=\prod^m_{j=1}\exp(A\Delta_j\tau)=\exp(A(b-a))$$
\end{proof}

\begin{example}
Let $A=\begin{pmatrix}0 & 1\\1 & 0\end{pmatrix}$. Then
$$\mint_0^1 \exp(A d(tI))=\exp(A)=\begin{pmatrix}\cosh(1) & \sinh(1)\\ \sinh(1) & \cosh(1)\end{pmatrix}$$
\end{example}

Unfortunately, there is no general formula for integrating a constant with respect to a general integrator. Typically we have
\begin{align}
\label{eqn:neqmintconst}
\mint_a^b \exp(A\,dE(t))\not=\exp(A\cdot (E(b)-E(a)))
\end{align}

\begin{example}
Let $A$ be as in the last example and $B=\mathrm{diag}(1,2)$. Then $AB~\not=~BA$. Set $E(t)=tA$ for $t\in[-1,0)$ and $E(t)=tB$ for $t\in [0, 1]$. Then
$$\mint_{-1}^1 \exp(dE(t))=e^A\cdot e^B\not= e^{A+B}=e^{E(1)-E(-1)}$$
\end{example}
But equality holds of course in \eqref{eqn:neqmintconst} if the family of matrices $\{A\}\cup\{E(t)\,:\,t\in[a,b]\}$ commutes.

\begin{prop}[Determinant formula]\index{determinant formula}
\label{prop:mintdet}
For $f\in\mathcal{M}^b_a[E]$ we have
$$\det\mint_a^b \exp(f(t)\,dE(t))=\exp\left(\int_a^b f(t)\,d\tr E(t)\right)$$
In particular, multiplicative integrals always yield invertible matrices.
\end{prop}

\begin{proof}
For $(\tau,\xi)\in\mathcal{T}^b_a$ we have
$$\det P(\tau,\xi)=\prod^m_{j=1} \det\exp(f(\xi_j)\Delta_j E) = \exp\left(\sum^m_{j=1} f(\xi_j) \Delta_j (\tr E)\right)$$
The claim now follows by continuity of the determinant.
\end{proof}

\begin{prop}
\label{prop:mintunitaryconst}
Let $f\in\mathcal{M}^b_a[E]$ and $U$ a constant invertible matrix. Then
$$U\mint^b_a \exp(f\,dE)U^{-1}=\mint^b_a\exp(f\,d(UEU^{-1}))$$
\end{prop} 

\begin{proof}
This follows from $Ue^AU^{-1}=e^{UAU^{-1}}$.
\end{proof}

The next fact will be useful to determine when a multiplicative integral is unitary.

\begin{prop}
\label{prop:mintgram}
Suppose $E$ is Hermitian and $f\in\mathcal{M}^b_e[E]$. Let
$$A=\mint_a^b \exp(f\,dE)$$
Then 
$$AA^*=\mint_a^b \exp(2\Re f\,dE)$$
\end{prop}

\begin{proof}
Let $(\tau,\xi)\in\mathcal{T}^b_a$. We have
$$\left(\prodr_{k=1}^m \exp(f(\xi_k)\Delta_k E)\right)^*=\prodl_{k=1}^m \exp(\overline{f(\xi_k)} \Delta_k E)$$
Further
$$\left(\prodr_{k=1}^m \exp(f(\xi_k)\,\Delta_k E)\right)\left(\prodl_{k=1}^m \exp(\overline{f(\xi_k)}\,\Delta_k E)\right)=\prodr_{k=1}^m \exp((f(\xi_k)+\overline{f(\xi_k)})\,\Delta_k E)$$
Letting $\nu(\tau)\rightarrow 0$ gives the claim.
\end{proof}

Now we prove two important estimates for multiplicative integrals in terms of additive Riemann-Stieltjes integrals.

\begin{prop}
\label{prop:mintest}
Assume that $f\in\mathcal{M}^b_a[E]$, $E\in\BV([a,b]; M_n)$ and $|f|$ is Stieltjes integrable with respect to $|E|$. Then
$$
\left\|\mint_a^b \exp(f(t)\,dE(t))\right\|\le \exp\left(\int^b_a |f(t)| d|E|(t)\right)$$
\end{prop}

\begin{proof}
Choose $(\tau,\xi)\in\mathcal{T}^b_a$. Then
$$\|P(\tau,\xi)\|\le \prod_{j=1}^m \exp(|f(\xi_j)| \|\Delta_j E\|)=\exp\left(\sum_{j=1}^m |f(\xi_j)| \|\Delta_j E\|\right).$$
Also, we have $\|\Delta_j E\|\le \var_{[t_{j-1},t_j]} E=|E|(t_j)-|E|(t_{j-1})=\Delta_j |E|$. Letting $\nu(\tau)\rightarrow 0$ we obtain the claim.
\end{proof}

\begin{prop}\index{multiplicative integral estimate}
\label{prop:minttaylorest}
Let $f\in\mathcal{M}^b_a[E]$, $E\in\BV([a,b]; M_n)$ and $|f|$ be Stieltjes integrable with respect to $|E|$. Then
$$
\mint^b_a \exp(f(t)\,dE(t))=I+\int_a^b f(t)\,dE(t)+R
$$
where $R$ is a matrix that satisfies
\begin{align}
\label{eqn:minttayloresterr}
\|R\|\le\sum_{\nu=2}^\infty \frac{1}{\nu!}\left(\int^b_a |f(t)|\,d|E|(t)\right)^\nu
\end{align}
If $\int_a^b |f| d|E|\le 1$, then there exists $0<C<1$ such that
$$\|R\|\le C\left(\int^b_a |f(t)|\,d|E|(t)\right)^2$$
\end{prop}

\begin{proof}
Let $(\tau,\xi)\in\mathcal{T}^b_a$. We expand the  product $P(\tau,\xi)$ using the exponential series:
$$P(\tau,\xi)=\prod_{j=1}^m \sum_{\nu=0}^\infty \frac{(f(\xi_j)\Delta_j E)^\nu}{\nu!}=I+\sum_{j=1}^m f(\xi_j)\Delta_j E+R$$
where the remainder term $R=R(\tau,\xi)$ is of the form $R=\sum_{\nu=2}^\infty T_\nu$ with the terms $T_\nu$ of order $\nu$ being
$$T_\nu=\sum_{\nu_1+\cdots+\nu_m=\nu}\prodr_{j=1}^m \frac{(f(\xi_j)\Delta_j E)^{\nu_j}}{\nu_j!}$$
We use the triangle inequality and $\|\Delta_j E\|\le \Delta_j |E|$ to estimate
\begin{align*}
\|T_\nu\|&\le \sum_{\nu_1+\cdots+\nu_m=\nu} \prod_{j=1}^m \frac{(|f(\xi_j)|\cdot \Delta_j |E|)^{\nu_j}}{\nu_j!}\\
&=\frac{1}{\nu!}\sum_{\nu_1+\cdots+\nu_m=\nu}{\nu\choose \nu_1,\cdots,\nu_m}\prod_{j=1}^m (|f(\xi_j)|\cdot\Delta_j |E|)^{\nu_j}\\
&=\frac{1}{\nu!}\left(\sum_{j=1}^m |f(\xi_j)| \cdot\Delta_j |E|\right)^\nu
\end{align*}
Letting $\nu(\tau)\rightarrow 0$ we obtain the first part of the claim.\\
If now $\int_a^b |f| d|E|\le 1$, then
$$\|R\|\le C\left(\int^b_a |f(t)|\,d|E|(t)\right)^2$$
holds with $C=\sum_{\nu=2}^\infty \frac{1}{\nu!}=e-2$.
\end{proof}

Lastly, we give a change of variables formula for multiplicative integrals. Note that we allow the variable transformation to have jump discontinuities. 

\begin{prop}[Change of variables\index{Change of variables}]
If $\varphi:[a,b]\rightarrow[\alpha,\beta]$ is a strictly increasing function with $\varphi(a)=\alpha$ and $\varphi(b)=\beta$, $f$ continuous on $[\alpha,\beta]$ and $E$ a continuous increasing mvf on $[a,b]$, then
$$
\mint^b_a \exp(f(\varphi(t))\,dE(t))=\mint^\beta_\alpha \exp(f(s)\,dE(\varphi^\dagger(s)))
$$
assuming that the the integrals exist.
\end{prop}

Here, $\varphi^\dagger$ is the \defin{generalized inverse}\index{generalized inverse} of $\varphi$ given by $\varphi^\dagger(s)=\inf\{t\in [a,b]\,:\,\varphi(t)\ge s\}$, which gives a natural notion of inverse for a general increasing function. For strictly increasing continuous functions we have $\varphi^\dagger=\varphi^{-1}$. Jump discontinuities of $\varphi$ translate into intervals of constancies of $\varphi^\dagger$ and intervals of constancy of $\varphi$ become jump discontinuities of $\varphi^\dagger$. The generalized inverse is normalized in the sense that it is left-continuous.

\begin{proof}
Set $g=f\circ\varphi$ and $F=E\circ\varphi^\dagger$. Given a tagged partition $(\tau, \xi)\in \mathcal{T}_{a}^{b}$ on $[a,b]$, we can apply $\varphi$ to produce a partition on $[a,b]$ given by $(\tau^\prime, \xi^\prime)=(\varphi(\tau), \varphi(\xi))\in \mathcal{T}_\alpha^\beta$. This mapping and its inverse respect the partial order on the set of tagged partitions. That is, refinements of $(\tau,\xi)$ map to refinements of $(\tau^\prime,\xi^\prime)$ and vice versa. Now it only remains to notice that
$$P(f,F,\tau^\prime,\xi^\prime)=\prod_{i=1}^m \exp(f(\xi_i^\prime)\Delta^{\tau^\prime}_i F)
=\prod_{i=1}^m \exp(g(\xi_i))\Delta^{\tau}_i E)=P(g, E, \tau,\xi)$$ 
and similarly the other way around.
\end{proof}

As opposed to the additive case, this formula does not extend to increasing functions which have intervals of constancy. This corresponds to the difficulties in computing the multiplicative integral of constants.

\subsection{The multiplicative Lebesgue integral}

The additive Lebesgue integral can be written as a Riemann-Stieltjes integral. We use this observation to define a multiplicative version of the Lebesgue integral and show that it behaves as expected. In particular, there is a Lebesgue differentiation theorem. Let us first recall some definitions. We denote the Lebesgue measure on the real line by $\lambda$.

We say that a mvf $A$ on $[a,b]$ is \defin{Lebesgue integrable}\index{Lebesgue integrable}, if
$$\int_a^b \|A(t)\|d\lambda(t)<\infty$$
in that case we write $A\in L^1([a,b]; M_n)$. As usual, we identify two such functions if they differ only on a set of measure zero. That is, $L^1$ functions are strictly speaking equivalence classes of functions.

A mvf $E$ on $[a,b]$ is \defin{absolutely continuous}\index{absolutely continuous} if for all $\epsilon>0$ there exists $\delta>0$ such that for all partitions $\tau=\{a=t_0<\dots<t_m=b\}$ with $\nu(\tau)<\delta$ we have
$$\var_{[a,b]}^\tau=\sum_{k=1}^m \|\Delta_k E\|<\epsilon$$

As in the scalar case, absolutely continuous mvfs are differentiable almost everywhere with integrable derivative. They are also by definition of bounded variation.

\begin{definition}[Multiplicative Lebesgue integral]
Let $A\in L^1([a,b]; M_n)$. Define
$$E(t)=\int^t_a A(s)\,d\lambda(s)$$
for $t\in[a,b]$. Then $E$ is absolutely continuous. We define the expression
$$\mint^b_a \exp(A(t)\,dt)=\mint^b_a \exp(dE(t))$$
to be the \defin{multiplicative Lebesgue integral}\index{multiplicative Lebesgue integral} of $A$. It is well-defined and also exists by Proposition \ref{prop:mintexist2}.
\end{definition}

As for ordinary Riemann-Stieltjes integrals, multiplicative integrals with an absolutely continuous integrator can be rewritten as multiplicative Lebesgue integrals.
\begin{prop} If $E$ is an absolutely continuous, increasing mvf and $f$ is a continuous scalar function, then
$$\mint_a^b \exp(f(t)dE(t))=\mint_a^b \exp(f(t)E^\prime(t)dt)$$
\end{prop}

\begin{proof}
As for scalar functions, $E$ is differentiable almost everywhere and the derivative $E^\prime$ is in $L^1([a,b]; M_n)$. Setting $F(t)=\int_a^t f(s)E^\prime(s)d\lambda(s)$,
the claim is equivalent to
\begin{align}
\label{eqn:lebesguepf1}
\mint_a^b\exp(dF(t))=\mint_a^b \exp(f(t)dE(t))
\end{align}
Pick a partition $\tau=\{a=t_0<\dots<t_m=b\}$. Apply the mean value theorem of integration to choose $\xi_j\in[t_{j-1},t_j]$ such that $f(\xi_j)\int_{t_{j-1}}^{t_j} E^\prime(s)d\lambda(s)=\int_{t_{j-1}}^{t_j} f(s)E^\prime(s)d\lambda(s)$. This is possible because $f$ is continuous. Then

\begin{align*}
\prod_{k=1}^m \exp(\Delta_j F)&=\prod_{k=1}^m \exp\left(\int_{t_{j-1}}^{t_j} f(s)E^\prime(s)d\lambda(s)\right)\\
&=\prod_{k=1}^m \exp\left(f(\xi_j)\int_{t_{j-1}}^{t_j}E^\prime(s)d\lambda(s)\right)\\
&=\prod_{k=1}^m \exp(f(\xi_j)\Delta_j E)
\end{align*}
In the limit $\nu(\tau)\rightarrow 0$, this implies \eqref{eqn:lebesguepf1}.
\end{proof}

Multiplicative integrals can also be viewed as solutions of certain ordinary differential equations\index{ordinary differential equation}, as we will see from the next proposition, which we can also view as a multiplicative version of the classical Lebesgue differentiation theorem\index{Lebesgue differentiation theorem}. 

\begin{prop}
Let $A\in L^1([a,b]; M_n)$. Then the function
$$F(x)=\mint_a^x \exp(A(t)dt)$$
is differentiable almost everywhere in $(a,b)$ and
\begin{align}
\label{mintode}
\frac{dF}{dx}(x)=F(x)A(x)
\end{align}
for almost every $x\in (a,b)$.
\end{prop}

\begin{remark}
The uniqueness and existence theory for ordinary differential equations shows that we could have also \emph{defined} the multiplicative integral $\mint_a^x \exp(A(t)dt)$ as the unique solution of the Cauchy problem\index{Cauchy problem} given by \eqref{mintode} and the initial value condition $F(a)=I_n$.
\end{remark}

\begin{proof}
By Propositions \ref{prop:mintsep} and \ref{prop:minttaylorest} we have for $x\in(a,b)$ and $h>0$ small enough:
$$F(x+h)=F(x)\mint_x^{x+h} \exp(A(t)dt)=F(x)(I+\int_x^{x+h} A(t)dt+R)$$
where $\|R\|\le C\left(\int_x^{x+h} \|A(t)\|dt\right)^2$. Therefore
$$\frac{1}{h}(F(x+h)-F(x))=F(x)\left(\frac{1}{h}\int_x^{x+h} A(t)dt+\frac{R}{h}\right)$$
A similar calculation works for $h<0$. By the Lebesgue differentiation theorem,
$$\lim_{h\rightarrow 0}\frac{1}{h}\int_x^{x+h} A(t)dt=A(x)$$
for almost every $x\in(a,b)$. Also,
$$\frac{\|R\|}{h}\le Ch\left(\frac{1}{h}\int_x^{x+h}\|A(t)\|dt\right)^2\longrightarrow 0$$
as $h\rightarrow 0$ for almost every $x\in (a,b)$. The claim follows.
\end{proof}

\subsection{Helly's theorems}
\label{hellythm}

Our theory of multiplicative integrals is still missing a convergence theorem. The aim of this section is to fill in this gap. The convergence theorem we will obtain is an analogue of Helly's convergence theorem for scalar Riemann-Stieltjes integrals (see Theorem \ref{thm:hellyconvsc} in the appendix). First, we prove two auxiliary statements.

A useful trick when estimating differences of products is the following.
\begin{lemma}[Telescoping identity]
\label{lemma:telescoping}\index{telescoping identity}
For matrices $Q_1,\dots,Q_m,P_1,\dots,P_m$ we have\begin{align}
\label{eqn:telescoping}
\prod^m_{\nu=1} P_\nu - \prod^m_{\nu=1} Q_\nu
=\sum_{l=1}^m \left(\prod^{l-1}_{\nu=1} P_\nu\right) (P_l-Q_l) \left(\prod^m_{\nu=l+1} Q_\nu\right)
\end{align}
\end{lemma}
\begin{proof} We do an induction on $m$. For $m=1$ there is nothing to show. Setting $A=\prod_{\nu=1}^{m-1} P_\nu$, $B=P_{m}$, $C=\prod_{\nu=1}^{m-1} Q_\nu$ and $D=Q_{m}$ we have
$$\prod_{\nu=1}^{m} P_\nu-\prod_{\nu=1}^{m} Q_\nu=AB-CD=(A-C)D+A(B-D)$$
By the induction hypothesis, the right hand side is equal to
$$\sum_{l=1}^{m} \left(\prod^{l-1}_{\nu=1} P_\nu\right) (P_l-Q_l) \left(\prod^{m}_{\nu=l+1} Q_\nu\right)$$
\end{proof}

The following is a simple estimate for the additive matrix-valued Riemann-Stieltjes integral.

\begin{lemma}\index{Stieltjes integral estimate}
\label{lemma:addstieltjesest}
If $E$ is an increasing mvf on $[a,b]$ and $f$ a bounded scalar function which is Stieltjes integrable with respect to $E$, then
$$\left\|\int_a^b f\,dE\right\|\le C \|E(b)-E(a)\|$$
for some constant $C>0$.
\end{lemma}

\begin{proof}
For $n=1$ the statement is true. Choose $C>0$ such that $|f(t)|\le C$ for $t\in[a,b]$. Let $v\in\C^n$. Since $v^*Ev$ is an increasing scalar function, 
\begin{align*}
\left|v^*\left(\int_a^b f\,dE\right)v\right| &=\left|\int_a^b f(t) d(v^*E(t)v)\right|\le \int_a^b |f(t)|d(v^*E(t)v)\\
&\le C (v^* (E(b)-E(a))v)
\end{align*}
Taking the supremum over all $v$ with $\|v\|=1$ we obtain the claim.
\end{proof}

\begin{thm}[Helly's convergence theorem for multiplicative integrals]\index{Helly's convergence theorem}\index{Helly's second theorem}
\label{thm:hellyconvmvf}
Let $(E_k)_k$ be a sequence of increasing and uniformly Lipschitz continuous mvfs on $[a,b]$ which converge pointwise to the function $E$ on $[a,b]$. Suppose also that $(f_k)_k$ is a sequence of uniformly bounded Riemann integrable scalar functions on $[a,b]$ which converges pointwise to the Riemann integrable function $f$. Then
\begin{align*}
\lim_{k\rightarrow\infty} \mint^b_a \exp(f_k\,dE_k)=\mint^b_a \exp(f\,dE)
\end{align*}
\end{thm}
The prerequisites in this theorem are of course not the mildest possible to arrive at the desired conclusion, but they allow a relatively easy proof and are sufficient for our applications.

\begin{proof}
Let us assume without loss of generality that $[a,b]=[0,1]$ and that the uniform Lipschitz constant for all the $E_k$ is $1$. Then, since $\|E_k(t)-E_k(s)\|\le |t-s|$ for all $k$, also the pointwise limit $E$ is Lipschitz continuous with Lipschitz constant $1$. In particular $f\in\mathcal{M}^1_0[E]$. Choose a constant $K>0$ such that $|f_k(t)|\le K$ and $|f(t)|\le K$ for all $t$ and $k$. Now we begin by estimating
$$\left\| \mint_0^1 e^{f\,dE}-\mint_0^1 e^{f_k\,dE_k}\right\|\le \left\|\mint_0^1 e^{f\,dE}- \mint_0^1 e^{f\,dE_k}\right\|+\left\|\mint_0^1 e^{f\,dE_k}-\mint_0^1 e^{f_k\,dE_k} \right\|$$
The differences on the right hand side we name $d_1$ and $d_2$, respectively. Let us first estimate $d_1$. By Proposition \ref{prop:mintest} we have

\begin{align}
\label{eqn:hellypf0}
\left\|\mint_a^b \exp(f \,dE_k)\right\|\le \exp\left(\int_a^b |f(t)| dt\right)\le M_0
\end{align}

for every $0\le a\le b\le 1$ and $k\in\N\cup\{\infty\}$, where we set $E_\infty=E$ and $M_0>0$ is a constant, not depending on $a,b$. Now let us choose a partition $\tau$ of $[0,1]$ by setting $t_i=\frac{i}{m}$ for $i=0,\dots,m$. Also set $I_i=[t_{i-1},t_i]$ for $i=1,\dots,m$.

We apply the telescoping identity \eqref{eqn:telescoping}\index{telescoping identity} and the previous estimate \eqref{eqn:hellypf0} to see
\begin{align}
\label{eqn:hellypf1}
d_1=\left\|\prod_{i=1}^m\mint_{t_{i-1}}^{t_i} e^{f\,dE}-\prod_{i=1}^m\mint_{t_{i-1}}^{t_i} e^{f\,dE_k}\right\|
\le M\sum_{i=1}^m \left\|\mint_{t_{i-1}}^{t_i} e^{f\,dE} -\mint_{t_{i-1}}^{t_i} e^{f\,dE_k}\right\|
\end{align}
where $M=M_0^2$.

For large enough $m$ we have $\int_{t_{i-1}}^{t_i} |f(t)| dt\le 1$. Then Proposition \ref{prop:minttaylorest} applied to the multiplicative integrals on the right hand side of \eqref{eqn:hellypf1} gives
\begin{align}
\label{eqn:hellypf2}
d_1 &\le M\sum_{i=1}^m \left\|\int_{t_{i-1}}^{t_i} f dE - \int_{t_{i-1}}^{t_i} f dE_k\right\|+2C\left(\int_{t_{i-1}}^{t_i} |f(t)| dt\right)^2
\end{align}
To see this, note that the Lipschitz condition on $E_k$ implies $|E_k|(t)-|E_k|(s)\le t-s$ for $t\ge s$ and $k\in\N\cup\{\infty\}$.

We estimate further

\begin{align}
\left\|\int_{t_{i-1}}^{t_i} f dE - \int_{t_{i-1}}^{t_i} f dE_k\right\| &\le
\left\|\int_{t_{i-1}}^{t_i} f dE-f(i/m)\Delta_i E\right\|+
\left\|f(i/m)\Delta_i E - f(i/m)\Delta_i E_k\right\|\nonumber\\
\label{eqn:hellypf3}
&+\left\|f(i/m)\Delta_i E_k-\int_{t_{i-1}}^{t_i} f dE_k\right\|
\end{align}

Using Lemma \ref{lemma:addstieltjesest}, we get
\begin{align}
\label{eqn:hellypf4}
\left\|\int_{t_{i-1}}^{t_i} f dE-f(i/m)\Delta_i E\right\|&=\left\|\int_{t_{i-1}}^{t_i}(f(t)-f(i/m))dE(t)\right\|\\
&\le (\osc_{I_i} f)\|\Delta_i E\|\le \frac{\osc_{I_i} f}{m}
\end{align}
By the same argument also
\begin{align}
\label{eqn:hellypf5}
\left\|f(i/m)\Delta_i E_k-\int_{t_{i-1}}^{t_i} f dE_k\right\|\le \frac{\osc_{I_i} f}{m}
\end{align}
Combining \eqref{eqn:hellypf3}, \eqref{eqn:hellypf4} and \eqref{eqn:hellypf5}, we see from \eqref{eqn:hellypf2}, that
\begin{align}
\label{eqn:hellypf6}
d_1\le MK\sum_{i=1}^m \|\Delta_i E-\Delta_i E_k\|+2M\sum_{i=1}^m (\osc_{I_i} f)\cdot \frac{1}{m}+\frac{2CK^2}{m^2}
\end{align}
Let $\epsilon>0$. Since $f$ is Riemann integrable, we can choose $m$ large enough such that
\begin{align}
\label{eqn:hellypf7}
2M\sum_{i=1}^m (\osc_{I_i} f)\cdot \frac{1}{m}+\frac{2CK^2}{m^2}\le\frac{\epsilon}{3}
\end{align}
With $m$ fixed like that we now use the pointwise convergence of $E_k$ and make $k$ large enough such that
\begin{align}
\label{eqn:hellypf8}
MK\sum_{i=1}^m \|\Delta_i E-\Delta_i E_k\|\le\frac{\epsilon}{6}
\end{align}

This is possible, since $E-E_k$ is only being evaluated at finitely many points. Whenever we make $m$ even larger later on during the estimate of $d_2$, we silently also increase $k$ accordingly such that \eqref{eqn:hellypf8} holds. 

Combining \eqref{eqn:hellypf6},\eqref{eqn:hellypf7} and \eqref{eqn:hellypf8}, we get $d_1\le\frac{\epsilon}{2}$.\\

Let us now estimate $d_2$. Similarly as for $d_1$, the telescoping identity and Proposition \ref{prop:minttaylorest} imply

\begin{align}
d_2 &\le M\sum_{i=1}^m \left\|\int_{t_{i-1}}^{t_i} f\,dE_k-\int_{t_{i-1}}^{t_i}  f_k\,dE_k\right\| +C\left(\int_{t_{i-1}}^{t_i} |f(t)|dt\right)^2+\\
&+C\left(\int_{t_{i-1}}^{t_i} |f_k(t)|dt\right)^2\nonumber
\le M\sum_{i=1}^m \int_{t_{i-1}}^{t_i} |f(t)-f_k(t)| dt+\frac{2CK^2}{m^2}\\\label{eqn:hellypf9}
&=M\int_0^1 |f(t)-f_k(t)|dt+\frac{2CK^2}{m^2}
\end{align}
Note that $f,f_k$ are measurable, therefore we may use Egorov's theorem\index{Egorov's theorem} to conclude that for every $\delta>0$, there exists a measurable set $Q\subset[0,1]$ such that $\lambda([0,1]\backslash Q)~\le~\delta$, where $\lambda$ denotes the Lebesgue measure on $[0,1]$, and $f_k\rightarrow f$ uniformly on $Q$. Let us choose $\delta=\frac{\epsilon}{12MK}$ and make $k$ large enough such that
$|f(t)-f_k(t)|\le\frac{\epsilon}{6M}$ for all $t\in Q$.
Then
\begin{align}
M\int_0^1 |f(t)-f_k(t)| dt &= M\int_Q |f-f_k| d\lambda+M\int_{[0,1]\backslash Q} |f-f_k| d\lambda\nonumber\\
\label{eqn:hellypf10}
&\le M\left(\frac{\epsilon}{6M}+2K\delta\right)=\frac{\epsilon}{3}
\end{align}
Note that we have reinterpreted the Riemann-Stieltjes integral as a Lebesgue integral.
Now \eqref{eqn:hellypf9} and \eqref{eqn:hellypf10} imply
$d_2 \le \frac{\epsilon}{2}$ for large enough $m$.

Altogether we proved that

$$\left\| \mint_0^1 e^{f\,dE}-\mint_0^1 e^{f_k\,dE_k}\right\|\le d_1+d_2\le \epsilon$$

for sufficiently large $k$. Since $\epsilon$ was arbitrary, the claim follows.
\end{proof}

We will now also prove an analogue of Helly's selection theorem (see Theorem \ref{thm:hellyselsc} in the appendix).
It is not directly related to multiplicative integrals, but it is a natural addendum to the previous theorem and important for the proof of Potapov's theorem in Section \ref{sect:contractivemvfs}.

\begin{thm}[Helly's selection theorem for matrix-valued functions]\index{Helly's selection theorem}\index{Helly's first theorem}
\label{thm:hellyselmvf}
Let $(E^{(k)})_k$ be a uniformly bounded sequence of increasing mvfs on $[a,b]$. Then there exists a subsequence $(E^{(k_j)})_j$ such that $E^{(k_j)}$ converges pointwise to an increasing mvf $E$ on $[a,b]$. 
\end{thm}

\begin{proof}
Choose $C>0$ such that $\|E^{(k)}(t)\|\le C$ for all $k$ and $t\in[a,b]$. We claim that for all $i,j$, the entries $E_{ij}^{(k)}$ form a sequence of BV-functions with uniformly bounded total variation. Let $\tau=\{a=t_0\le t_1\le \cdots\le t_m=b\}$ be a partition of $[a,b]$. Then we estimate
\begin{align}
\label{eqn:hellyselmvfpf1}
|\Delta_{l} E_{ij}^{(k)}|\le \|\Delta_l E^{(k)}\|\le \tr \Delta_l E^{(k)}=\Delta_l \tr E^{(k)}
\end{align}
for $l=1,2,\dots,m$. In the second inequality we have used positivity of $\Delta_l E^{(k)}$. Therefore
$$\var_{[a,b]}^\tau E_{ij}^{(k)}=\sum_{l=1}^m |\Delta_{l} E_{ij}^{(k)}|\le \sum_{l=1}^m \Delta_l \tr E^{(k)}=\tr (E^{(k)}(b)- E^{(k)}(a))$$
Using the estimate $\tr A=\sum_{i=1}^n A_{ii}\le n\|A\|$ for $A\ge 0$ we see that
$$\var_{[a,b]} E_{ij}^{(k)}\le 2nC$$
Hence we can repeatedly apply Helly's scalar selection Theorem \ref{thm:hellyselsc} to find a subsequence such that all the entries $E^{(k_j)}_{ij}$ converge pointwise to BV-functions $E_{ij}$. The resulting mvf $E=(E_{ij})_{ij}$ is also increasing as pointwise limit of increasing mvfs.
\end{proof}

Both theorems were given by Potapov in \cite[Appendix \S 2]{Potapov}, but with a slightly different proof for the selection theorem.

\section{Contractive analytic mvfs}
\label{sect:contractivemvfs}

By an \defin{analytic mvf}\index{analytic mvf} on the unit disk $\D\subset\C$ we mean a function $A:\D\rightarrow M_n$ all components of which are holomorphic throughout $\D$.
A mvf $A:\D\rightarrow M_n$ is called \defin{contractive}\index{contractive mvf} if $$A(z)A^*(z)\le I$$ for all $z\in\D$, i.e. the Hermitian matrix $I-A(z)A^*(z)$ is positive semidefinite. An equivalent condition is that $\|A(z)\|\le 1$ for all $z\in\D$. This and some other basic facts are proven in Appendix \ref{sect:background}.

We say that $A$ is \defin{bounded}\index{bounded mvf} if
$$\|A\|_\infty=\sup_{z\in\D}\|A(z)\|<\infty$$
We denote the space of bounded analytic matrix functions on $\D$ whose determinant does not vanish identically by $\H^\infty$\index{$\H^\infty$}. The subspace of analytic mvfs, which are also contractive on $\D$ will be denoted by $\Sch\subset\H^\infty$\index{$\Sch$}\index{Schur class}, where the letter $\Sch$ is chosen in honour of I. Schur, who studied this class of functions in the case $n=1$.

The goal of this section is to prove a theorem of V.P. Potapov on the multiplicative structure of functions in the class $\Sch$. In his paper \cite{Potapov}, Potapov considered the more general case of $J$-contractive mvfs\index{$J$-contractive mvf}, however we only require the case $J=I$. 

\begin{definition} Let
\begin{align}
\label{betadef}
\beta_{z_0}(z)=\left\{\begin{array}{cc}
\frac{z_0-z}{1-\overline{z_0}z}\frac{|z_0|}{z_0}, & \text{if}\,z_0\not=0\\
z, & \text{if}\,z_0=0
\end{array}\right.
\end{align}
A \defin{Blaschke-Potapov factor} (B.P. factor)\index{Blaschke-Potapov factor}\index{B.P. factor} is a function $b:\D\rightarrow M_n$ such that
\begin{align}
\label{BPfact1}
b(z)=U\left(\begin{array}{cc}\beta_{z_0}(z) I_r & 0\\0 & I_{n-r}\end{array}\right)U^*
\end{align}
for all $z\in\D$, where $U$ is a constant unitary matrix, $0<r\le n$ and $z_0\in\C$. We call $r$ the \defin{rank} of the B.P. factor. A \defin{Blaschke-Potapov product} (B.P. product)\index{Blaschke-Potapov product}\index{B.P. product} is a possibly infinite product of B.P. factors, which may be multiplied from the right or left by a constant unitary matrix. By convention, also a unitary constant is considered a B.P. product.
\end{definition}

Equivalently, we can define a B.P. factor by
\begin{align}
\label{BPfact2}
b(z)=I-P+\beta_{z_0}(z)P
\end{align}
where $P\in M_n$ is an orthogonal projection and $z_0\in\D$. The connection to \eqref{BPfact1} is that $P$ projects onto the $r$-dimensional image of $I-b(0)$, i.e. 
$P=U\left(\begin{array}{cc}I_r & 0\\0 & 0\end{array}\right)U^*$. This shows that B.P. factors are uniquely determined by the zero $z_0$ and the choice of the subspace that $P$ projects onto.

The condition for the convergence of a B.P. product turns out to be the same condition as for the scalar analogue (see Theorem \ref{bpconv}). Whenever in the following we speak of a B.P. product as a function, it is understood implicitly that the B.P. product really is convergent in the sense defined later in this section.

Let us denote the Herglotz kernel\index{Herglotz kernel} by
$$h_z(\theta)=\frac{z+e^{i\theta}}{z-e^{i\theta}}$$
for $z\in\D$ and $\theta\in[0,2\pi]$.

\begin{thm}[Potapov]\index{Potapov's fundamental theorem}
\label{thm:potapov}
Let $A\in\Sch$. Then $A$ can be written as
\begin{align}
\label{eqn:potapovrepr}
A(z)=B(z)\cdot \mint^L_0 \exp\left(h_z(\theta(t))\,dE(t)\right)
\end{align}
Here, $B$ is a B.P. product corresponding to the zeros of $\det A$, $0\le L<\infty$, $E$ an increasing mvf such that $\tr E(t)=t$ for $t\in[0,L]$ and $\theta:[0,L]\rightarrow [0,2\pi]$ a right-continuous increasing function.
\end{thm}

The proof of the theorem will proceed in two steps. First we detach a maximal Blaschke-Potapov product to obtain a mvf with non-vanishing determinant. In the second step we use an approximation by rational functions and Helly's theorems to obtain the desired multiplicative integral representation.

\subsection{Blaschke-Potapov products}
In this section we discuss the convergence of B.P. products and prove a factorization of an arbitrary function in $\H^\infty$ into a maximal Blaschke-Potapov product and a function with non-vanishing determinant. By a \defin{zero}\index{zero of a mvf} of a function $A\in\Sch$, we always mean a zero of $\det A$, i.e. a point at which $A$ becomes singular. If $z$ is a zero of $A$, then the dimension of $\ker A$ is called the \defin{rank}\index{rank of a zero} of the zero.

\subsubsection{Convergence of B.P. products}
A scalar Blaschke product $b(z)=z^m\prod^\infty_{i=1}\frac{z_i-z}{1-\overline{z_i}z}\frac{|z_i|}{z_i}$ converges if and only if the \defin{Blaschke condition}\index{Blaschke condition}
$$\sum_{i\ge 1} (1-|z_i|)<\infty$$
is fulfilled (see \cite[Chapter II.2]{Garnett}). We will prove that the same is true for B.P. products. It is clear that the Blaschke condition is necessary. For, if $A\in\H^\infty$, then $\det A$ is a scalar bounded analytic function, whence it follows that its zeros satisfy the Blaschke condition.

We will now prove the converse by adapting the corresponding scalar proof as presented in \cite{Rudin2}.

Let us call an infinite product $P=\prod^\infty_{i=1} B_i$ of matrices $B_i$ \defin{convergent}\index{convergent product} if the sequence given by $P_k=\prod^k_{i=1}B_i=B_1\cdot B_2\cdots B_k$ converges in the $\|\cdot\|$ norm. The limit is then denoted by $\prod^\infty_{i=1}B_i$. The notion of uniform convergence for an infinite product of matrix-valued functions is defined accordingly.

\begin{lemma}\label{estimatelemma}
Let $(B_i)_{i=1,\dots,k}$ be given matrices. Then
$$\|I-\prod_{i=1}^k B_i\|\le \prod^k_{i=1}\left(1+\|I-B_i\|\right)-1$$
\end{lemma}
\begin{proof}
We do an induction on $k$. For $k=1$ there is nothing to show. Writing $A_k=\prod_{i=1}^k B_i$ and $a_k=\prod_{i=1}^k (1+\|I-B_i\|)$ we see that
$$A_{k+1}-I=(A_k-I)(I+(B_{k+1}-I))+B_{k+1}-I$$
and thus
$$\left\|I-A_{k+1}\right\|
\le (a_k-1)(1+\|I-B_{k+1}\|)+\|I-B_{k+1}\|=a_{k+1}-1$$
by the induction hypothesis.
\end{proof}

\begin{thm}\label{thm:prodconv}Let $(B_i)_i$ be matrix-valued functions on $\D$ such that
\begin{enumerate}
\item[(a)] $\sum_{i\ge 1}\|I-B_i(z)\|$ converges uniformly on compact subsets of $\D$ and
\item[(b)] the sequence $P_k(z)=\prod^k_{i=1}B_i(z)$ is uniformly bounded on $\D$.
\end{enumerate}
Then $P(z)=\prod^\infty_{i=1}B_i(z)$ converges uniformly on compact subsets of $\D$. The same theorem holds if we replace $\prod$ everywhere by $\prodl$.
\end{thm}

\begin{proof}
Let $K\subset\D$ be compact. By assumption there is $C>0$ such that $\|P_k(z)\|\le C$ for all $z\in\D$ and $k\ge 1$. Choose $\epsilon>0$ and $N$ so large that
$$\sup_{z\in K}\sum^l_{i=k+1}\|I-B_i(z)\|<\epsilon$$
for $l\ge k\ge N$. Then we have by Lemma \ref{estimatelemma}
\begin{align*}
\|P_k(z)-P_l(z)\| &\le \|P_k(z)\|\cdot\|I-\prod_{i=k+1}^l B_i(z)\|\\
&\le 
C\left(\prod^l_{i=k+1}(1+\|I-B_i(z)\|)-1\right)\\
&\le C\left(\exp\left(\sum^l_{i=k+1}\|I-B_i(z)\|\right)-1\right)\\
&\le C(e^\epsilon-1)
\end{align*}
for $z\in K$. The right hand side converges to $0$ if $\epsilon\rightarrow 0$. The proof for $\prodl$ works analogously using also an analogous version of Lemma \ref{estimatelemma}.
\end{proof}

\begin{thm}\label{bpconv}\index{Blaschke condition}
If $B$ is a formal B.P. product corresponding to the zeros $z_1,z_2,\dots$ and the Blaschke condition holds, then $B$ converges (uniformly on compact sets) to an analytic matrix function in $\D$ and $\|B\|_\infty=1$.
\end{thm}

\begin{proof}
We may assume $z_i\not=0$. Now we want to apply Theorem \ref{thm:prodconv}. Let $B_k(z)=\prod^k_{i=1}b_i(z)$ be the $k$th partial product of $B$ and 
$$b_i(z)=U_i\left(\begin{array}{cc}\frac{z_i-z}{1-\overline{z_i}z}\frac{|z_i|}{z_i} I_{r_i} & 0\\0 & I_{n-r_i}\end{array}\right)U_i^{-1}=I-U_i
\left(\begin{array}{cc}\left(1-\frac{z_i-z}{1-\overline{z_i}z}\frac{|z_i|}{z_i}\right) I_{r_i} & 0\\0 & 0\end{array}\right)U_i^{-1}$$
Then for $|z|\le r<1$ we get
$$\|I-b_i(z)\|=\left|1-\frac{z_i-z}{1-\overline{z_i}z}\frac{|z_i|}{z_i}\right|
=\left|\frac{z_i+z|z_i|}{z_i-z|z_i|^2}\right|(1-|z_i|)\le\frac{1+r}{1-r}(1-|z_i|)$$
which implies condition (a). 
Note that $\|B_k\|$ is the absolute value of a finite scalar Blaschke product, so condition (b) is satisfied. By the theorem and Lemma \ref{unifconvlemma} we obtain $B$ as an analytic matrix function. $\|B\|_\infty=1$ is again clear because $\|B\|$ is the absolute value of a scalar Blaschke product.
\end{proof}

\subsubsection{Factorization}
Our next objective is the factorization of an arbitrary bounded analytic matrix function into a B.P. product and a function with non-vanishing determinant. 
To do this we will first analyze the detachment of a single Blaschke-Potapov factor from a given contractive analytic function.

\begin{definition}If $A\in\Sch$ has a zero at some $z_0\in\D$, we call a B.P. factor $b$ with $b(z_0)=0$ \defin{detachable} from $A$ if $b^{-1}A\in\Sch$.
\end{definition}

\begin{lemma}[Detachability condition]\index{detachability condition}
\label{detachlemma}
Suppose $A\in\Sch$ and $\det A(z_0)=0$ for $z_0\in\D$. Then a B.P.factor $b$, which is given by \eqref{BPfact1} is detachable from $A$ if and only if 
\begin{align}
\label{detachcrit}
U^*A(z_0)=\left(\begin{array}{cc}0_r & 0\\ * & * \end{array}\right)
\end{align}
where $0_r$ denotes the $r\times r$ zero matrix and the $*$ denote arbitrary block matrices of the appropriate dimensions.
\end{lemma}

\begin{proof}
We use Taylor development to write
$$A(z)=A(z_0)+R(z)(z-z_0)$$
for $z$ in some neighborhood of $z_0$ which is small enough to be contained in $\D$ and $R$ a holomorphic mvf.
Writing $U^*A(z_0)=\left(\begin{array}{cc}B & C\\ * & * \end{array}\right)$ for some $r\times r$ matrix $B$ and $r\times (n-r)$ matrix $C$ we get
\begin{align*}
b^{-1}(z)A(z) &= U\left(\begin{array}{cc}\beta_{z_0}(z) I_r & 0\\0 & I_{n-r}\end{array}\right)\left(\begin{array}{cc}B & C\\ * & * \end{array}\right)+U\left(\begin{array}{cc}\beta_{z_0}(z)(z-z_0)I_r & 0\\0 & I_{n-r}\end{array}\right)U^* R(z)\\
&= \left(\begin{array}{cc}\beta_{z_0}(z)B & \beta_{z_0}(z)C\\ * & * \end{array}\right)+U\left(\begin{array}{cc}\beta_{z_0}(z)(z-z_0)I_r & 0\\0 & I_{n-r}\end{array}\right)U^* R(z)
\end{align*}
Recalling that $\beta_{z_0}^{-1}$ is given by \eqref{betadef} and thus only has a simple pole at $z_0$ we see that $b^{-1}A$ is holomorphic at $z_0$ if and only if \eqref{detachcrit} holds. It remains to show that $b^{-1}A$ is in that case contractive. Let $\epsilon>0$. 
Then we can choose $0\le r<1$ so close to $1$ that 
$$\|b^{-1}(z)\|=|\beta_{z_0}(z)|^{-1}=\left|\frac{1-\overline{z_0}z}{z_{0}-z}\right|\le 1+\epsilon$$ for $|z|\ge r$. This can be seen from the properties of the pseudohyperbolic distance as described in \cite[Chapter I.1]{Garnett}. Since $A$ is contractive by assumption this implies
$$\|b^{-1}(z)A(z)\|\le 1+\epsilon$$
for $|z|\ge r$. Because the norm of an analytic mvf is subharmonic\index{subharmonic} (see the appendix for a proof of this), we conclude by the maximum principle\index{maximum principle} (see \cite[Theorem I.6.3]{Garnett}) that this estimate holds also for $|z|<r$. As $\epsilon$ was arbitrary, we obtain $\|b^{-1}(z)A(z)\|\le 1$ for every $z\in\D$.
\end{proof}

Using the alternative formulation \eqref{BPfact2} for $b$, the detachability condition\index{detachability condition} \eqref{detachcrit} becomes
\begin{align}
\label{detachcrit2}
\mathrm{Im} A(z_0)\subseteq \ker P
\end{align}
or equivalently, $\Img A(z_0)\perp\Img P$. If we also require the rank $r$ of $P$ to be maximal, the condition becomes 
\begin{align}
\label{maxdetachcrit}
\Img A(z_0)=\ker P
\end{align}
so the B.P. factor is in that case uniquely determined.

We are now ready to prove the main result of this section.
\begin{thm}\label{thm:bpfactor}\index{B.P. product}
Given $A\in\Sch$, there exists a B.P. product $B$ and $\tilde{A}\in\Sch$ without zeros, such that $A=B\cdot\tilde{A}$.

Moreover, $B$ is uniquely determined up to multiplication with a constant unitary matrix.
\end{thm}
The uniqueness statement says that $B$ is uniquely determined as a function on $\D$, \emph{not} as a formal product. That is, the individual B.P. factors may be quite different depending on the order in which we detach the zeros.

{\em Proof of existence.} Let $z_1,z_2,\dots$ be the zeros of $\det A$ in no particular order, counted according to their multiplicities. Let $0\le r<1$. By the Blaschke condition, there can only be finitely many zeros such that $|z_i|\le r$. We now construct sequences of B.P. factors $(b_k)_k$ and functions $(A_k)_k$ in $\Sch$ by the following inductive process starting with $A_0=A$ and $k=1$:\\

If $\det A_{k-1}$ has a zero at $z_k$ then $A_{k-1}(z_k)$ has defect $0< r_k\le n$ and by singular value decomposition we obtain
\begin{align}
\label{svd}
A_{k-1}(z_k)=U_k\left(\begin{array}{cc}
0 & 0\\
0 & D\\
\end{array}\right)V_k
\end{align}
where $D$ is a $(n-r_k)\times(n-r_k)$ diagonal matrix with non-zero entries and $U_k$ and $V_k$ are unitary matrices.
Now set
\begin{align}
\label{defbkak}
b_k(z)=U_k\left(\begin{array}{cc}
\beta_{z_k}(z)I_{r_k} & 0\\
0 & I_{n-r_k}\end{array}\right)U_k^{-1}\,\,\,\mbox{and}\,\,\,
A_k(z)=b^{-1}_k(z)A_{k-1}(z)
\end{align}
for $z\in\D$. By Lemma \ref{detachlemma}, this effects the detachability of $b^{-1}_k$ from $A_{k-1}$, so $A_k\in\Sch$.\\

Now we continue from the start with $k+1$ instead of $k$, where we skip those $k$ such that $\det A_{k-1}(z_k)\not=0$, which may happen since the individual zeros occur as often as their multiplicity dictates. From the equation
$$\det A_k(z)=\beta_{z_k}(z)^{r_k} \det A_{k-1}(z)$$
we see that each zero $z_i$ will be ``consumed" eventually, i.e. there exists $N=N(i)$ such that $\det A_k(z_i)\not=0$ for $k\ge N$. Also, the process will end after finitely many steps if and only if there are finitely many zeros.
In the case of infinitely many zeros, we know from Theorem \ref{bpconv} that $B_k(z)=\prod^k_{i=1}b_i(z)$ will converge to a B.P. product $B$. We claim that also the sequence $(A_k)_k$ converges to a bounded analytic function $\tilde{A}$. For the proof note that \eqref{defbkak} implies
$$A_k(z)=\left(\prodl_{i=1}^k b_i(z)^{-1}\right)A(z)$$
Also a calculation shows that for $|z|\le r<1$
$$\|I-b_k(z)^{-1}\|=\left|1-\frac{1-\overline{w}z}{w-z}\frac{w}{|w|}\right|
=\left|\frac{w+z|w|}{w|w|-z|w|}\right|(1-|w|)\le \frac{1}{1-r}\cdot(1-|w|)$$
where $w=z_k\not=0$. Hence we can apply Theorem \ref{thm:prodconv} to conclude convergence of the partial products $A_k$ against a function $\tilde{A}\in\Sch$ which satisfies $A=B\cdot\tilde{A}$.\qed

To prove uniqueness we require the following observation.
\begin{lemma}
\label{detachlemma2}
Let $A,A_1,A_2\in\Sch$ satisfy $A(z)=A_1(z)A_2(z)$. Furthermore, suppose that $\det A_1$ has a zero at $z_0\in\D$ and $\det A_2(z_0)\not=0$. Let $b$ be a Blaschke-Potapov factor for $z_0$, which is detachable from $A$ in the sense defined above. Then it is also detachable from $A_1$.
\end{lemma}

\begin{proof}The claim follows from Lemma \ref{detachlemma} and
$$
U^*A_1(z_0)=\left(\begin{array}{cc}0_r & 0\\ * & *\end{array}\right)A_2^{-1}(z_0)=\left(\begin{array}{cc}0_r & 0\\ * & *\end{array}\right)$$
\end{proof}
\noindent Now we can complete the proof of Theorem \ref{thm:bpfactor}.

{\em Proof of uniqueness.} Suppose that
$$A=B^{(1)}\cdot \tilde{A}^{(1)}=B^{(2)}\cdot \tilde{A}^{(2)}$$
are two factorizations such that for $i=1,2$, $B^{(i)}$ is a B.P. product and $\tilde{A}^{(i)}$ a contractive analytic mvf with non-vanishing determinant. Without loss of generality we write
$$B^{(i)}(z)=\prod^\infty_{j=1}b^{(i)}_j(z)\;\;\;\text{and}\;\;\;
B^{(i)}_k(z)=\prod^k_{j=1}b^{(i)}_j(z)\;\;\;(z\in\D)$$
for $i=1,2$, where the $b_j^{(i)}$ are B.P. factors. In case the $B^{(i)}$ come with constant unitary right factors, we can include those in the $\tilde{A}^{(i)}$. Clearly, $b^{(2)}_{k+1}$ is detachable from
$$B_k^{(2)}(z)^{-1} B^{(2)}(z)=\prod_{j=k+1}^\infty b^{(2)}_j(z)$$
By Lemma \ref{detachlemma2}, it is also detachable from 
$$B_k^{(2)}(z)^{-1} B^{(1)}(z)=B_k^{(2)}(z)^{-1} B^{(2)}(z)\tilde{A}^{(2)}(z)\tilde{A}^{(1)}(z)^{-1}$$
Letting $k\rightarrow\infty$ we obtain that $F=(B^{(2)})^{-1}B^{(1)}\in\Sch$. By symmetry, also $F^{-1}=(B^{(1)})^{-1}B^{(2)}$ is contractive. Hence we have $I-F(z)^{-1}F^*(z)^{-1}\ge 0$ for $z\in\D$, so also 
$$0 \le F(z)(I-F(z)^{-1}F^*(z)^{-1})F^*(z)=F(z)F^*(z)-I$$ 
But at the same time we know $I-F(z)F^*(z)\ge 0$, so $F$ must be unitary everywhere in $\D$. By Corollary \ref{cor:unitaryconst} in the appendix, $F$ is a constant unitary matrix. This concludes the proof of Theorem \ref{thm:bpfactor}.\qed

\subsubsection{Finite B.P. products}
A consequence of the above factorization is the following characterization of finite B.P. products which turns out to be the same as in the scalar case.

\begin{lemma}\index{B.P. product}
\label{lemma:finitebp}
A mvf $A\in\Sch$ is a finite Blaschke-Potapov product if and only if it extends continuously to $\overline{\D}$ and takes unitary values on $\T$.
\end{lemma}

\begin{proof}
Suppose that $A\in\Sch$ extends continuously to $\overline{\D}$ and takes unitary values on $\T$. The Blaschke condition implies that $\det A$ has only finitely many zeros, since otherwise they would accumulate at some point on $\T$, which is impossible because $\|A\|=1$ on the unit circle. By Theorem \ref{thm:bpfactor} there exists a finite Blaschke-Potapov product $B$ and $\tilde{A}\in\Sch$ such that $A=B\cdot\tilde{A}$ and $\det\tilde{A}$ is non-vanishing. Hence also $\tilde{A}^{-1}\in\Sch$ and $\tilde{A}^{-1}=B^{-1}\cdot A$ also extends continuously to $\overline{\D}$ with unitary values on $\T$. In particular, $\|\tilde{A}^{-1}(z)\|=1$ for $z\in\T$. By the maximum principle for subharmonic functions, $\tilde{A}^{-1}$ is contractive.
Summarizing, we have for arbitrary $z\in\D$
$$I-\tilde{A}(z)\tilde{A}^*(z)\ge 0\,\,\,\text{and}\,\,\,I-\tilde{A}(z)^{-1}(\tilde{A}^*(z))^{-1}\ge 0$$
whence it follows that $\tilde{A}$ is unitary at $z$. This implies that $\tilde{A}$ is equal to a constant unitary matrix which proves the claim.
The other implication follows directly from the definition of a B.P. factor.
\end{proof}

\subsubsection{B.P. products for $n=2$} In the case of $2\times 2$ matrices, the B.P. factors can have only rank $1$ or $2$. B.P. factors of rank $2$ are just scalar Blaschke factors times the identity matrix. Thus we can factor out a maximal scalar Blaschke product to obtain a function which has only zeros of rank 1.

\begin{lemma}\index{B.P. product}
Let $n=2$ and $A\in\Sch$. Assume that $A$ has no zeros of rank 2. Let $z_1,z_2,\dots$ be an enumeration of the zeros of $A$ counted with multiplicities. Then there exist uniquely determined B.P. factors $b_1,b_2,\dots$ and $\tilde{A}\in\Sch$ without zeros such that $b_k(z_k)=0$ and
$$A=\prod_{k=1}^N b_k(z)\cdot \tilde{A}(z)$$
where $N\in\N\cup\{\infty\}$ is the number of zeros (including multiplicities).
\end{lemma}

This is clear in view of the above. In fact, condition \eqref{maxdetachcrit} allows for the $b_k$ to be expressed explicitly: Let
$$b_k(z)=I-P_k+\beta_{z_k}(z)P_k$$
be the $k$th B.P. factor and $A_k$ be recursively given by
$$A_1(z)=A(z)\,\,\text{and}\,\,A_{k+1}(z)=b_k^{-1}(z)A_k(z)\,\,\text{for}\;k\ge 1$$
Then $b_k$ is determined by $\ker P_k=\Im A_k(z_k)$.

\subsection{Multiplicative representation}
The key ingredient for obtaining a multiplicative representation of a contractive analytic matrix function with non-vanishing determinant will be the following approximation theorem. The proof is from \cite[Chapter V]{Potapov}.

\begin{thm}\index{rational approximation}
\label{thm:rationalapprox}
For every $A\in\Sch$ there exists a sequence of rational contractive mvfs $(A_k)_{k\ge 1}$, unitary on $\T$, such that $A_k$ converges to $A$ uniformly on compact sets in $\D$ as $k\rightarrow\infty$.
\end{thm}
By a rational mvf, we mean a matrix-valued functions the entries of which are scalar rational functions. 
Combined with Lemma \ref{lemma:finitebp}, this theorem implies that every contractive analytic matrix function can be uniformly approximated by finite Blaschke-Potapov products.

\begin{proof}
We may choose a constant $w\in\C$ with $|w|=1$ such that $\det(wI-A(0))\not=0$ since $A(0)$ can have at most $n$ distinct eigenvalues. This implies that the holomorphic function $\det(wI-A(z))$ is not identically zero, so the matrix $wI-A(z)$ is regular for all but countably many points $(\mu_j)_j$ in $\D$. Thus we may define a (possibly meromorphic) mvf by
\begin{align}
\label{Cayley}
T(z)=i(wI-A(z))^{-1}(wI+A(z))
\end{align}
This transformation is a matrix-valued analogue of a conformal mapping from the unit disk to the upper half plane which is sometimes referred to as Cayley transform\index{Cayley transform}. The inverse transformation is given by
\begin{align}
\label{CayleyInv}
A(z)=w(T(z)-iI)(T(z)+iI)^{-1}
\end{align}
A calculation shows
$$\Im T(z)=(wI-A(z))^{-1}(I-A(z)A^*(z))(\overline{w}I-A^*(z))^{-1}\ge 0$$
because $I-A(z)A^*(z)\ge 0$ by assumption. By the open mapping principle, we know that a scalar meromorphic function that maps onto the upper-half plane is actually holomorphic. The same holds for meromorphic mvfs with non-negative imaginary part, because the diagonal entries are scalar functions mapping onto the upper-half plane and and the off-diagonal entries can be bounded in terms of diagonal entries. Consequently, the singularities $(\mu_j)_j$ are removable and $T$ is holomorphic. By the Herglotz representation theorem (see the appendix for a proof of this theorem) we can write $T$ as $$T(z)=T_0+i\int^{2\pi}_0 \frac{e^{it}+z}{e^{it}-z}\,d\sigma(t)$$
where $T_0$ is a constant Hermitian matrix and $\sigma$ an increasing mvf. This integral representation allows us to approximate $T$ uniformly by Riemann-Stieltjes sums.

Choose $0<r_k<1$ for $k=1,2,\dots$ such that $r_k\nearrow 1$ as $k\rightarrow\infty$ and none of the points $\mu_j$ lies on any of the circles $\{z:|z|=r_k\}$. For each $k=1,2,\dots$ we also pick an appropriate subdivision $0\le t^{(k)}_0\le t^{(k)}_1\le \dots \le t^{(k)}_{m_k} = 2\pi$ of the interval $[0,2\pi]$ such that the Riemann-Stieltjes sum
\begin{align}
\label{Tkdef}
T_k(z)=T_0+i\sum^{m_k-1}_{\nu=0} \frac{e^{it^{(k)}_\nu}+z}{e^{it^{(k)}_\nu}-z}(\sigma(t^{(k)}_{\nu+1})-\sigma(t^{(k)}_\nu))
\end{align}
satisfies the estimate
\begin{align}
\label{TTkestimate}
\|T(z)-T_k(z)\|\le \frac{1}{k}\hspace{1cm}\text{for all}\,|z|\le r_k
\end{align}
By construction, the rational functions $(T_k)_k$ are holomorphic on $\D$ and converge to $T$ uniformly on compact sets. The meromorphic function 
\begin{align}
\label{TpIeq}
T(z)+iI=2iw(wI-A(z))^{-1}
\end{align}
is only singular at the points $(\mu_j)_j$. Hence $\det(T_k(z)+iI)$ can at most vanish at countably many points for large enough $k$. Thus it makes sense to define 
\begin{align}
\label{Akdef}
A_k(z)=w(T_k(z)-iI)(T_k(z)+iI)^{-1}
\end{align}
From \eqref{Tkdef} we see that
\begin{align}
\label{ImTkeq}
\Im T_k(z)=\sum^{m_k-1}_{\nu=0}\frac{1-|z|^2}{|e^{it^{(k)}_\nu}-z|}(\sigma(t^{(k)}_{\nu+1})
-\sigma(t^{(k)}_\nu))
\end{align}
which implies $\Im T_k(z)\ge 0$. It follows that
$$I-A_k(z)A_k^*(z)=4(T_k(z)+iI)^{-1}\cdot\Im T_k(z)\cdot(T_k^*(z)-iI)^{-1}\ge 0$$
so $A_k$ is a rational contractive mvf which has no poles in $\D$. Also, \eqref{ImTkeq} implies that $A_k(z)$ is unitary for $|z|=1$. We claim that $A_k$ converges to $A$ uniformly on compact sets. To prove this, let $K\subset\D$ be compact and $N$ large enough such that $K$ is contained in the disk $|z|\le r_N$. By choice of the $(r_k)_k$ we may select a $\delta>0$ such that none of the singularities $(\mu_j)_j$ lie in the annulus $R=\{z:r_N-\delta\le|z|\le r_N+\delta\}\subset\D$. 
Then, by the identity 
$$A(z)-A_k(z)=2i(T(z)+iI)^{-1}(T(z)-T_k(z))(T_k(z)+iI)^{-1}$$
and using \eqref{TTkestimate}, \eqref{TpIeq} we obtain 
\begin{align}
\label{AAkestimate1}
\|A(z)-A_k(z)\|\le \frac{2}{k}\cdot \|(T_k(z)-iI)^{-1}\|
\end{align}
for $z\in R$ and $k$ so large that $T_k(z)-iI$ is invertible in $R$ and $\{z:|z|~\le~r_k\}~\supset~R$. Since $T_k(z)-iI$ converges uniformly to $T(z)-iI$, the matrix norm $\|T_k(z)-iI\|$ is bounded uniformly in $k$ and $z$. By the matrix norm estimate $\|A^{-1}\|~\le~\frac{\|A\|^{n-1}}{|\det A|}$ (see the appendix for a proof of this) it follows that
\begin{align}
\label{Tkinvest}
\|(T_k(z)-iI)^{-1}\|\le \frac{C_0}{|\det(T_k(z)-iI)|}
\end{align}
for some large enough $C_0>0$. Since $\det(T_k(z)-iI)\not=0$ in $R$ for large enough $k$ and also $\det(T(z)-iI)\not=0$ in $R$ we can infer that $|\det(T_k(z)-iI)|$ is bounded from below by some positive constant. Combining this with \eqref{AAkestimate1} and \eqref{Tkinvest} we get
\begin{align}
\label{AAkestimate2}
\|A(z)-A_k(z)\|\le \frac{C}{k}
\end{align}
for some large enough $C>0$ and $z\in R$. From Lemma \ref{subharmlemma} and the maximum principle for subharmonic functions we conclude that \eqref{AAkestimate2} holds throughout $\{z:|z|\le r_N+\delta\}\supset K$.
\end{proof}

We will now describe how to obtain the multiplicative representation, proceeding as in \cite[Introduction, Chapter V]{Potapov}. Let $A\in\Sch$ have no zeros. Applying Theorem \ref{thm:rationalapprox} and Lemma \ref{lemma:finitebp}, we can choose a sequence
\begin{align}
\label{Akeq}
A_k(z)=b_1^{(k)}(z)b_2^{(k)}(z)\cdots b_{m_k}^{(k)}(z)U_k
\end{align}
where the $b_j^{(k)}$, $1\le j\le m_k$ are B.P. factors and the $U_k$ unitary matrices, such that $A_k\rightarrow A$ uniformly on compact subsets. Let also
\begin{align}
\label{bjkequ1}
b_j^{(k)}(z) &= U^{(k)}_j\left(\begin{array}{cc}\beta_{z^{(k)}_j}(z) I_{r^{(k)}_j} & 0\\0 & I_{n-r^{(k)}_j}\end{array}\right)(U^{(k)}_j)^*
\end{align}
with $U^{(k)}_j$ unitary and\footnote{We tacitly suppose that $z^{(k)}_j\not=0$. This is true anyway for large enough $k$.} $z^{(k)}_j=\rho^{(k)}_j e^{i\theta^{(k)}_j},\,j=1,2,\dots,m_k$ with $\rho^{(k)}_j>0$ and $0\le\theta^{(k)}_j<2\pi$. We may assume that the $z^{(k)}_j$ are arranged in order of increasing $\theta^{(k)}_j$. Now define
\begin{align}
\label{Hdef}
H^{(k)}_j=U^{(k)}_j\left(\begin{array}{cc} (1-|z_j^{(k)}|)I_{r^{(k)}_j} & 0\\0 & 0\end{array}\right)(U^{(k)}_j)^*
\end{align}
Notice that $\|H_j^{(k)}\|=1-|z_j^{(k)}|$. Now \eqref{betadef}, \eqref{bjkequ1} and \eqref{Hdef} imply
\begin{align}
\label{bjkequ2}
b_j^{(k)}(z)=I-\frac{z_j^{(k)}+|z_j^{(k)}|z}{z_j^{(k)}-|z_j^{(k)}|^2 z}H_j^{(k)}
\end{align}
The sequence $(\det A_k(0))_k$ converges, hence it is bounded by some constant $L>0$. Since
$$\prod^{m_k}_{\nu=1}|z_\nu^{(k)}|^{r_\nu^{(k)}}=\det A_k(0)$$
there exists a constant $C>0$ such that
\begin{align}
\label{Cinequ}
\sum^{m_k}_{\nu=1}(1-|z_\nu^{(k)}|)\le C
\end{align}
We would like to write the right hand side of \eqref{Akeq} as a multiplicative integral. However, this is not possible because multiplicative integrals are invertible everywhere while $\det A_k(z)$ has zeros. To remedy this situation, we consider the modified factors
\begin{align}
\label{bjktildedef}
\tilde{b}_j^{(k)}(z)=\exp\left(-\frac{e^{i\theta_j^{(k)}}+z}{e^{i\theta_j^{(k)}}-z}H_j^{(k)}\right)=\exp(h_z(\theta_j^{(k)})H_j^{(k)})
\end{align}
and accordingly
\begin{align}
\label{Aktildedef}
\tilde{A}_k(z)=\tilde{b}_1^{(k)}(z)\cdots\tilde{b}_{m_k}^{(k)}(z)U_k
\end{align}
We now show that the non-vanishing of $\det A$ implies that we may work with $\tilde{A}_k$ instead.
\begin{lemma}
\label{lemma:Atildelemma}
With $A_k,\tilde{A_k}$ given as above, we have that $\|A_k-\tilde{A}_k\|\rightarrow 0$ uniformly on compact sets in $\D$.
\end{lemma}

\begin{proof}
Let $0<r<1$ and let $k$ be large enough such that none of the points $(z_j^{(k)})_j$ lie in the disk $\{z:|z|\le r\}$. By \eqref{bjkequ2} we can estimate
\begin{align}
\label{bjkest}
\|b_j^{(k)}(z)\|\le 1+\frac{1+r}{1-r}\|H_j^{(k)}\|\le e^{\frac{1+r}{1-r}\|H_j^{(k)}\|}=e^{\frac{1+r}{1-r}(1-|z_j^{(k)}|)}
\end{align}
for $|z|\le r$. The same conclusion holds with $b_j^{(k)}$ replaced by $\tilde{b}_j^{(k)}$. Now the telescoping identity\index{telescoping identity} from Lemma \ref{lemma:telescoping} implies via \eqref{bjkest} and \eqref{Cinequ} that
\begin{align}
\label{AAtildeest}
\|A_k(z)-\tilde{A}_k(z)\|=e^{C\frac{1+r}{1-r}}\sum_{\nu=1}^{m_k}\|b_\nu^{(k)}(z)-\tilde{b}^{(k)}_\nu(z)\|
\end{align}
for $|z|\le r$. Define
$$c_j^{(k)}(z)=I-\frac{e^{i\theta_j^{(k)}}+z}{e^{i\theta_j^{(k)}}-z}H_j^{(k)}$$
We write
$$
\|b_j^{(k)}(z)-\tilde{b}^{(k)}_j(z)\|\le \|b_j^{(k)}(z)-c_j^{(k)}(z)\|
+\|c_j^{(k)}(z)-\tilde{b}^{(k)}_j(z)\|
$$
and estimate the two differences separately. First,
\begin{align}
\label{bjkcjkest}
\|b_j^{(k)}(z)-c_j^{(k)}(z)\| &= \left|
\frac{e^{i\theta}+z}{e^{i\theta}-\rho z}-
\frac{e^{i\theta}+z}{e^{i\theta}-z}
\right|\cdot \|H_j^{(k)}\|\notag\\
&\le 2\frac{1-\rho}{(1-r)^2}\|H_j^{(k)}\|=2\left(\frac{1-|z_j^{(k)}|}{1-r}\right)^2
\end{align}
where $\rho=\rho_j^{(k)}=|z_j^{(k)}|$, $\theta=\theta_j^{(k)}$ and $|z|\le r$. Secondly,
\begin{align}
\label{cbtildeest}
\|c_j^{(k)}(z)-\tilde{b}_j^{(k)}(z)\| &=\left\|\sum_{\nu=2}^\infty \frac{(h_z(\theta_j^{(k)})H_j^{(k)})^\nu}{\nu!}\right\|\le \sum_{\nu=2}^\infty 
\frac{1}{\nu!}\left(\frac{1+r}{1-r}\right)^\nu \|H_j^{(k)}\|^\nu\notag\\
&\le 4e^{C\frac{1+r}{1-r}}\left(\frac{1-|z_j^{(k)}|}{1-r}\right)^2
\end{align}
where $|z|\le r$.
Setting
$$
M=M(r)=\frac{2}{(1-r)^2}\cdot e^{C\frac{1+r}{1-r}}\max\{1,2e^{C\frac{1+r}{1-r}}\}
$$
and applying \eqref{bjkcjkest}, \eqref{cbtildeest} to \eqref{AAtildeest}, we see
\begin{align}
\label{AAtildeest2}
\|A_k(z)-\tilde{A}_k(z)\|\le M \sum_{\nu=1}^{m_k} (1-|z_\nu^{(k)}|)^2 \le CM \max_{\nu=1,\dots,m_k} (1-|z_\nu^{(k)}|)
\end{align}
in the disk $\{z:|z|\le r\}$. Since $A_k\rightarrow A$ and $\det A$ has no zeros in $\D$, the right hand side of \eqref{AAtildeest2} converges to $0$ as $k\rightarrow\infty$.
\end{proof}
Now we will write \eqref{Aktildedef} as a multiplicative Stieltjes integral on some interval $[0,L]$. Define $t_0^{(k)}=0$ and
\begin{align}
\label{tjkdef}
t_j^{(k)}=\sum^{j}_{\nu=1}\tr H_\nu^{(k)}\;\;\;\text{for}\;j=1,\dots,m_k
\end{align}
From \eqref{tjkdef} and \eqref{Hdef} we see
\begin{align*}
t_j^{(k)}\le\sum_{\nu=1}^{m_k} r_\nu^{(k)}(1-|z_\nu^{(k)}|)<
\prod^{m_k}_{\nu=1}|z_\nu^{(k)}|^{r_\nu^{(k)}}=\det A_k(0)\le L
\end{align*}
Finally, we define
\begin{align}
\label{Ekdef}
E^{(k)}(t)=\left\{\begin{array}{ll}\sum^{j-1}_{\nu=1} H^{(k)}_\nu + \frac{t-t^{(k)}_{j-1}}{t_j^{(k)}-t_{j-1}^{(k)}}H_j^{(k)}, & \text{if}\;t_{j-1}^{(k)}\le t<t_j^{(k)}\;\text{for some}\;j=1,\dots,m_k\\
\sum^{m_k}_{\nu=1}H_\nu^{(k)}, & \text{if}\;t_{m_k}^{(k)}\le t\le L\end{array}\right.\\
\label{thetadef}
\theta^{(k)}(t)=\theta^{(k)}_j\;\text{for}\;t_{j-1}^{(k)}\le t<t_j^{(k)}\;\;\text{and}\;\;\theta^{(k)}(t)=\theta^{(k)}_{m_k}\;\text{for}\;
t_{m_k}^{(k)}\le t\le L
\end{align}
$E^{(k)}$ is chosen such that it is an increasing mvf on $[0,L]$ and the equation
$$\tr E^{(k)}(t)=t$$
is satisfied for $t\in[0,t_{m_k}]$. Note also that for $0\le s\le t\le L$ we have
$$\|E^{(k)}(t)-E^{(s)}(s)\|\le \tr (E^{(k)}(t)-E^{(k)}(s)) \le t-s$$
so $E^{(k)}$ is Lipschitz continuous.
The function $\theta^{(k)}:[0,L]\rightarrow[0,2\pi)$ is increasing and right-continuous. Thus, $h_z(\theta^{(k)}(t))$ is Riemann integrable on $[0,L]$ as a function of $t$.  By Proposition \ref{prop:mintconst}, \eqref{bjktildedef} and \eqref{Ekdef} we can write
$$
\tilde{b}_j^{(k)}(z) = \mint_{t_{j-1}^{(k)}}^{t_j^{(k)}}
\exp\left(\frac{h_z(\theta^{(k)}(t))}{t_j^{(k)}-t_{j-1}^{(k)}}H_j^{(k)}dt\right)
=\mint_{t_{j-1}^{(k)}}^{t_j^{(k)}} 
\exp\left(h_z(\theta^{(k)}(t))\,dE^{(k)}(t)\right)$$
for $j=1,\dots,m_k$. We write all the factors on the right hand side of \eqref{Aktildedef} in a single multiplicative integral using Proposition \ref{prop:mintsep} to obtain \begin{align}
\label{Akmint}
\tilde{A}_k(z)=\mint^L_0 \exp\left(h_z(\theta^{(k)}(t))\,dE^{(k)}(t)\right)
\end{align}
Now we apply Helly's selection Theorem\index{Helly's selection theorem} \ref{thm:hellyselmvf} twice to extract a common subsequence such that both, $(E^{(k_j)})_j$ and $(\theta^{(k_j)})_j$, converge to respective limit functions $E$ and $\theta$. Then $E$ is an increasing mvf with $\|E(t)-E(s)\|\le t-s$ for $t\ge s$. Moreover, $\theta$ is a bounded increasing function. In particular, it has only countably many discontinuities. Thus $h_z(\theta(t))$ is Riemann integrable in $t$ on $[0,L]$. By Helly's convergence theorem\index{Helly's convergence Theorem} for multiplicative integrals (Theorem \ref{thm:hellyconvmvf}) we conclude that
$$\lim_{j\rightarrow\infty}\tilde{A}_{k_j}(z)= \mint_0^L \exp(h_z(\theta(t))\,dE(t))$$
But by Lemma \ref{lemma:Atildelemma}, $\tilde{A}_{k_j}(z)$ converges also to $A(z)$. Consequently,

$$A(z)=\mint_0^L \exp(h_z(\theta(t))\,dE(t))$$

Should $\theta$ not be right-continuous we can change its values at the discontinuities such that it will be right-continuous. Since this changes $\theta$ only at countably many places, the value of the integral stays the same.

We have obtained the desired multiplicative representation and thereby finished the proof of Theorem \ref{thm:potapov}.

\section{Inner-outer factorization}
\label{sect:innerouterfact}

\subsection{Existence}

 The inner-outer factorization of a bounded analytic matrix-valued function was discovered by Y.P. Ginzburg in \cite{Ginzburg}, where he also indicates the basic steps for the proof. It should be noted that these factorization theorems can be achieved in a much more general setting without the use of multiplicative integrals. This was done for example by Sz.-Nagy and Foias in \cite{Sz.-Nagy}. However, their treatment is quite abstract and our goal here are  specifically the explicit multiplicative representations.

Let us first remark that for a bounded analytic matrix function $A$, it makes sense to speak of its values almost everywhere on the circle, defined for instance radially by
$$A(e^{i\theta})=\lim_{r\rightarrow 1-}A(re^{i\theta})$$
This limit exists for a.e. $\theta\in[0,2\pi)$ since the components of $A$ are bounded analytic scalar functions. We will denote the radial limit\index{radial limit} function also by $A$.

\begin{definition}\label{def:inner}
A function $A\in\H^\infty$ is called \defin{inner}\index{inner function}, if $A$ is unitary almost everywhere on $\T$. An inner function without zeros is a \defin{singular inner}\index{singular inner function} function. 
\end{definition}
By the maximum principle for subharmonic functions applied to $\|A\|$, inner functions are contractive.

In the scalar case, we can split up singular inner functions further with respect to the decomposition of the corresponding singular measure into pure point and singular continuous components. In the matrix-valued case this leads to the following definitions.
\begin{definition}
A mvf $A\in\H^\infty$ is called \defin{pp-inner}\index{pp-inner} (short for pure point inner), if there exist a unitary constant $U$, $m\in\N\cup\{\infty\}$, $l_k>0$, $\theta_k\in[0,2\pi)$ and increasing mvfs $E_k$ with $\tr E_k(t)=t$ such that
\begin{align}
\label{eqn:ppinner}
A(z)=U\prod_{k=1}^m \mint_0^{l_k} \exp(h_z(\theta_k)\,dE_k(t))
\end{align}
for $z\in\D$.
\end{definition}

A pp-inner function is inner. It suffices to check this for the case that $A(z)~=~\mint_0^{l} \exp(h_z(\theta)\,dE(t))$ for a constant $\theta\in[0,2\pi)$ and an increasing mvf $E$. By Proposition \ref{prop:mintgram} we have
$$A(z)A^*(z)=\mint_0^{l} \exp(2\Re h_z(\theta)\,dE(t))$$
Set $z=re^{i\varphi}$. If $\varphi\not=\theta$, then $\lim_{r\rightarrow 1-}\Re h_z(\theta)=0$. Therefore, the radial limit $A(e^{i\varphi})$ is unitary. 

We call a mvf \defin{singular continuous}\index{singular continuous} if it is nonconstant, continuous, increasing and has derivative $0$ almost everywhere.

\begin{definition}
A mvf $A\in\H^\infty$ is called \defin{sc-inner}\index{sc-inner} (short for singular continuous-inner), if there exist a unitary constant $U$ and a singular continuous mvf $S$ such that
\begin{align}
\label{eqn:scinner}
A(z)=U\mint_0^{2\pi} \exp(h_z(\varphi)dS(\varphi))
\end{align}
for $z\in\D$.
\end{definition}

Let us check that an sc-inner function is really an inner function. Applying Proposition \ref{prop:mintgram} and Proposition \ref{prop:minttaylorest} we have

$$A(z)A^*(z)=I+2\int_0^{2\pi} \Re h_z(\varphi) dS(\varphi)+R$$
where the error term $R$ satisfies the estimate  \eqref{eqn:minttayloresterr}. But

$$\left\|\int_0^{2\pi} \Re h_z(\varphi) dS(\varphi)\right\|\le \int_0^{2\pi} |\Re h_z(\varphi)| d|S|(\varphi)$$

The estimate \eqref{eqn:hellyselmvfpf1} applied to $S$ gives

$$\int_0^{2\pi} |\Re h_z(\varphi)| d|S|(\varphi)\le \int_0^{2\pi} |\Re h_z(\varphi)| d\tr S(\varphi)$$

Combining the last three estimates with \eqref{eqn:minttayloresterr} we get

\begin{align}
\label{eqn:scinnerisinnereq1}
\|I-A(z)A^*(z)\| \le \exp\left(\int_0^{2\pi} |\Re h_z(\varphi)| d\tr S(\varphi)\right)-1
\end{align}

Now notice that $\tr S$ is a scalar singular continuous function. Set $z=re^{i\theta}$. By the scalar theory (compare \cite{Rudin2}), the left hand side of \eqref{eqn:scinnerisinnereq1} converges to $0$ for almost every $\theta$ as $r$ approaches $1$. We conclude that sc-inner functions are really inner functions.

\begin{definition}
We call $A\in\H^\infty$ \defin{outer}\index{outer function} if
\begin{align}
\label{eqn:outer}
A(z)=U\mint^{2\pi}_0 \exp(h_z(\varphi)M(\varphi)\,d\varphi)
\end{align}
where $U$ is a unitary constant and $M\in L^1([0,2\pi]; M_n)$ a Hermitian mvf, whose least eigenvalue is bounded from below.
\end{definition}

Later, we will give equivalent characterizations for outer functions. Let us remark that all these definitions agree in the case $n=1$ with the corresponding scalar concepts.

Our goal is a canonical factorization of any $A\in\H^\infty$ into a B.P. product, a pp-inner, an sc-inner and an outer function. To achieve this, we will manipulate the multiplicative integral representation obtained from Potapov's fundamental Theorem \ref{thm:potapov}. We need two lemmas.

The first one says roughly that we can commute two functions in $\Sch$ as long as we only care about their determinants. This was given by Ginzburg, but unfortunately he did not provide a proof.
\begin{lemma}
\label{lemma:ginzburg}
Let $A_1,A_2\in\Sch$. Then there exist $\tilde{A}_1,\tilde{A}_2\in\Sch$ such that 
$$A_1\cdot A_2=\tilde{A}_2\cdot\tilde{A}_1$$
and $\det A_j=\det\tilde{A}_j$ for $j=1,2$.
\end{lemma}

\begin{proof}
By Theorem \ref{thm:rationalapprox} and Lemma \ref{lemma:finitebp}, we can choose a sequence of finite B.P. products
$$A_2^{(k)}(z)=b^{(k)}_1(z)\cdots b^{(k)}_{m_k}(z)$$
such that $A_2^{(k)}$ converges to $A_2$ uniformly on compact sets. Let us assume without loss of generality that $A_1$ and $A_2^{(k)}$ have no common zeros. From the mvf $A_1A_2^{(k)}$, we detach a B.P. product corresponding to the zeros of $A_2^{(k)}$ which we call $\tilde{A}_2^{(k)}$. That is, we obtain
\begin{align}
\label{ginzburglemma1}
A_1 A_2^{(k)}=\tilde{A}_2^{(k)} \tilde{A}_1^{(k)}
\end{align}
where $\tilde{A}_2^{(k)}$ is a finite B.P. product such that $\det A_2^{(k)}=\det \tilde{A}_2^{(k)}$ and $\tilde{A}_1^{(k)}\in\Sch$ the remainder.
By Montel's theorem\index{Montel's theorem} (see Theorem \ref{thm:montel}), there exists a subsequence $(\tilde{A}^{(k_i)}_2)_i$ which converges on compact sets to some analytic mvf $\tilde{A}_2$. We also get $\tilde{A}_2\in\Sch$ and $\det\tilde{A}_2=\det A_2$.
The identity \eqref{ginzburglemma1} implies that also $\tilde{A}^{(k_m)}_1$ converges to some $\tilde{A}_1\in\Sch$ with $\det A_1=\det\tilde{A}_1$.
\end{proof}

The next lemma says that the angular function $\theta$ in Potapov's multiplicative representation \eqref{eqn:potapovrepr} is already determined by $\det A$. The point of the proof is to use uniqueness of the measures in the scalar inner-outer factorization.

\begin{lemma}
\label{lemma:theta}
Suppose that $A\in\Sch$ and
\begin{align}
\label{eqn:lemma2pf1}
\det A(z)=\exp\left(\int_0^L h_z(\theta(t))dt\right)
\end{align}
for $z\in\D$, where $L>0$ and $\theta:[0,L]\rightarrow[0,2\pi]$ is a right-continuous, increasing function. Then
$$A(z)=U\mint_0^L\exp(h_z(\theta(t))dE(t))$$
for some unitary constant $U$ with $\det U=1$ and an increasing mvf $E$ with $\tr E(t)=t$.
\end{lemma}

\begin{proof}
By Potapov's Theorem\index{Potapov's fundamental theorem} \ref{thm:potapov} we know that 
\begin{align}
\label{eqn:lemma2pf2}
A(z)=U\mint_0^{\tilde{L}} \exp(h_z(\tilde{\theta}(t))dE(t))
\end{align} 
for a right-continuous, increasing function $\tilde{\theta}$ and an increasing mvf $E$ with $\tr E(t)=t$. Taking the determinant we see
\begin{align}
\label{eqn:lemma2pf3}
\det A(z)=\det(U)\exp\left(\int_0^{\tilde{L}} h_z(\tilde{\theta}(t))dt\right)
\end{align}
Plugging in the value $z=0$ and comparing with \eqref{eqn:lemma2pf1} yields $\det U=1$ and $\tilde{L}=L$.
Now it suffices to show that $\theta$ in the representation \eqref{eqn:lemma2pf1} is uniquely determined. 

By a change of variables we get
\begin{align}
\label{eqn:thetauniquepf2}
\int^{L}_0 h_z(\theta(t))\,dt=\int^{2\pi}_0 h_z(\varphi)\,d\theta^\dagger(\varphi)
\end{align}
where $\theta^\dagger$ is the left-continuous generalized inverse\index{generalized inverse} of $\theta$ given by $\theta^\dagger(\varphi)=\inf\{t\in[0,L]\,:\,\theta(t)\ge \varphi\}$. Let $\mu$ be the unique positive Borel measure such that $\int_0^{2\pi} f(\varphi) d\theta^\dagger(\varphi)=\int_0^{2\pi} f(\varphi) d\mu(\varphi)$ for all continuous functions $f$. Note that the map $\theta\mapsto \theta^\dagger\mapsto \mu$ is injective.

But by the scalar inner-outer factorization we know that the measure $\mu$ in
$$\det A(z)=\exp\left(\int_0^{2\pi} h_z(\varphi) d\mu(\varphi)\right)$$
is uniquely determined (see \cite{Garnett}, \cite{Rudin2}). Therefore also $\theta$ in \eqref{eqn:lemma2pf1} is uniquely determined.

\end{proof}

\begin{cor}
\label{cor:thetacont}
Let $A\in\Sch$ and assume
$$\det A(z)=\exp\left(\int_0^{2\pi} h_z(\varphi) d\psi(\varphi)\right)$$
for some continuous increasing function $\psi$. Then
$$A(z)=U\mint_0^{2\pi} \exp(h_z(\varphi) dE(\psi(\varphi))$$
for some unitary constant $U$ with $\det U=1$ and an increasing mvf $E$ satisfying $\tr E(t)=t$.
\end{cor}
\begin{proof}
Since $\psi$ is continuous, it is the generalized inverse of some right-continuous and strictly increasing $\theta$. Changing variables, applying Lemma \ref{lemma:theta} and changing variables again yields the claim.
\end{proof}

The following observation lies in the same vein.

\begin{lemma}
\label{lemma:detinnerouter}
Let $A\in\Sch$. Then $A$ is outer (resp. sc-inner) if and only if $\det A$ is outer (resp. sc-inner).
\end{lemma}

\begin{proof}
If $A$ is outer or sc-inner, respectively, then $\det A$ is by the determinant formula also outer or sc-inner, respectively.

If on the other hand $\det A$ is outer or sc-inner, respectively, then we find

$$\det A(z)=c\cdot\exp\left(\int_0^{2\pi} h_z(\varphi)\,d\psi(\varphi)\right)$$

with $|c|=1$ and $\psi$ being an absolutely continuous or singular continuous increasing function, respectively. Therefore we can conclude by Corollary \ref{cor:thetacont} that
$$A(z)=U\exp\left(\int_0^{2\pi} h_z(\varphi)\,dE(\psi(\varphi))\right)$$
where $U$ is a unitary constant and $E$ an increasing mvf satisfying $\tr E(t)~=~t$. Note that $E$ is Lipschitz continuous, because for $t\ge s$ we have 
$$\|E(t)-E(s)\|\le \tr (E(t)-E(s))=t-s$$
It follows that if $\psi$ is absolutely continuous, then also $E\circ\psi$ is absolutely continuous. On the other hand, if $\psi$ is singular continuous, then $E\circ\psi$ is also singular continuous. That is, $A$ is outer or sc-inner, respectively.
\end{proof}

Now we can proceed to proving Theorem \ref{thm:main}. In the first step we show how to detach the pp-inner factor.

\begin{lemma}
\label{lemma:ppinnerfact}
Let $A\in\Sch$ and assume that $A$ has no zeros. Then there exists a pp-inner function $S_{pp}$ and a continuous increasing mvf $\Sigma$ such that
\begin{align}
A(z)=S_{pp}(z)\cdot\mint_0^{2\pi} \exp(h_z(\varphi)d\Sigma(\varphi))
\end{align}
for all $z\in\D$.
\end{lemma} 

\begin{proof}
By Potapov's Theorem \ref{thm:potapov}
\begin{align}
A(z)=U\mint^L_0 \exp\left(h_z(\theta(t))\,dE(t)\right)
\end{align}
where $U$ is a unitary constant, and $E,L,\theta$ are as stated in that theorem. Let $\{(a_k,b_k)\,:\,k=1,\dots,m\}$ with $m\in\N\cup\{\infty\}$ be a complete enumeration of all the intervals on which $\theta$ is constant and let $\theta_k$ be the value of $\theta$ on the interval $(a_k,b_k)$. The length of the $k$th interval will be denoted by $\ell_k=b_k-a_k$. For any increasing function $\psi$, we denote the total length of intervals on which $\psi$ is constant by $\ell(\psi)$. That is, $\ell(\theta)=\sum_{i=1}^m \ell_i$. Note that $\ell(\theta)\le L<\infty$.

We inductively construct sequences $(S_j)_j, (A_j)_j$ of contractive mvfs such that $S_0=I$, $A_0=A$, $S_j\cdot A_j=A$ and
\begin{align}
\label{eqn:ppinnerfact1}
S_j(z)=\prod_{i=1}^j\mint_0^{\ell_i} \exp(h_z(\theta_i)dE_i(t)),\,\,A_j(z)=U_j\mint_0^{L_j}\exp(h_z(\theta^{(j)}(t))dF_j(t))
\end{align}
for $j\ge 1$ with $U_j$ unitary constants, $L_j=L-\sum_{i=1}^{j-1} \ell_i$, $\theta^{(j)}$ increasing functions and $F_j, E_j$ increasing mvfs satisfying $\tr E_j(t)=\tr F_j(t)=t$. We also want the intervals of constancy of $\theta^{(j)}$ to have exactly the lengths $\ell_{j+1},\ell_{j+2},\dots,\ell_m$ and that $\theta^{(j)}$ assumes the values $\theta_{j+1},\theta_{j+2},\dots,\theta_m$ on them, respectively.

Let $0\le k< m$ and assume we have constructed $A_j,S_j$ already for all $0\le j\le k$. Then $A_{k}$ has an interval of constancy of length $\ell_{k+1}$, say $(a,b)$. We rewrite $A_k$ as
$$A_k(z)=\mint_0^a e^{h_z(\theta^{(k)}(t))dF_{k}(t)} \cdot \mint_a^b e^{h_z(\theta_{k+1})dF_{k}(t)} \cdot \mint_b^{L_{k}}e^{h_z(\theta^{(k)}(t))dF_{k}(t)}$$

By Lemma \ref{lemma:ginzburg}, we can interchange the first two factors while preserving their determinants. Using Lemma \ref{lemma:theta} and Proposition \ref{prop:mintunitaryconst} on the modified factors yields
\begin{align}
\label{eqn:ppinnerfact3}
A_k(z)=\mint_0^{\ell_{k+1}} e^{h_z(\theta_{k+1})dG(t)}\cdot \tilde{U}\cdot\mint_0^a e^{h_z(\theta^{(k)}(t))dH(t)} \cdot \mint_b^{L_{k}}e^{h_z(\theta^{(k)}(t))dF_{k}(t)}
\end{align}
where $\tilde{U}$ is a unitary constant with $\det\tilde{U}=1$ and $G,H$ are increasing mvfs with $\tr G(t)=\tr H(t)=t$. Now we define $E_{k+1}=G$ and define $S_{k+1}$ as prescribed in \eqref{eqn:ppinnerfact1}. Also set
\begin{align}
\label{eqn:ppinnerfact4}
A_{k+1}(z)=\tilde{U}\cdot\mint_0^a e^{h_z(\theta^{(k)}(t))dH(t)} \cdot \mint_b^{L_{k}}e^{h_z(\theta^{(k)}(t))dF_{k}(t)}
\end{align}
Now it remains to check that $A_{k+1}$ has the form prescribed in \eqref{eqn:ppinnerfact1}. Computing the determinant gives
\begin{align*}
\det A_{k+1}(z) &=\exp\left(\int_0^a h_z(\theta^{(k)}(t))dt + \int_b^{L_{k}}h_z(\theta^{(k)}(t))dt\right)\\
&=\exp\left(\int_0^{L_{k+1}} h_z(\theta^{(k+1)}(t)) dt\right)
\end{align*}
where we have set $L_{k+1}=L_k-b+a$ and
$$\theta^{(k+1)}(t)=\left\{\begin{array}{ll}
\theta^{(k)}(t),\, & \text{for}\;t\in[0,a)\\
\theta^{(k)}(t+b-a),\, & \text{for}\;t\in[a,L_{k+1}]
\end{array}\right.
$$
By Lemma \ref{lemma:theta}, $A_{k+1}$ can be written as
$$A_{k+1}(z)=U_{k+1}\mint_0^{L_{k+1}}\exp(h_z(\theta^{(k+1)}(t))dF_{k+1}(t))$$
where $U_{k+1}$ is a unitary constant and $F_{k+1}$ an increasing mvf with $\tr F_{k+1}(t)~=~t$. 

If $m<\infty$, then this process terminates after $m$ steps. Then $A=S_mA_m$ is the desired factorization up to the unitary constants $U_k$ which we can pull up to the front using Proposition \ref{prop:mintunitaryconst}. So from now on we assume $m=\infty$. We claim that the infinite product
$$S_\infty(z)=\prod_{i=1}^\infty \mint_0^{\ell_i} \exp(h_z(\theta_i)dE_i(t))$$
converges. 
To verify that, we invoke Theorem \ref{thm:prodconv}. Condition (b) is trivially satisfied because all the factors are contractive mvfs. Denote the factors by $B_i(z)=\mint_0^{\ell_i} \exp(h_z(\theta_i)dE_i(t))$.  To check condition (a), we write
$$B_i(z)=I+\int_0^{\ell_i} h_z(\theta_i) dE_i(t)+R_i$$
where the remainder term $R_i$ satisfies \eqref{eqn:minttayloresterr}. Suppose that $K\subset\D$ is compact. Then there is $C>0$ such that $\frac{1+r}{1-r}\le C$ for all $z=re^{i\varphi}\in K$. Now estimate 
$$\left\|\int_0^{\ell_i} h_z(\theta_i) dE_i(t)\right\|\le \int_0^{\ell_i} |h_z(\theta_i)| dt \le C \ell_i$$
Let $N$ be large enough such that $C\ell_i\le 1$ for all $i\ge N$.  Then $R_i\le C^2 \ell_i^2$ for $i\ge N$ and
$$\sum_{i=N}^\infty \|I-B_i(z)\|\le \sum_{i=N}^\infty \int_0^{\ell_i}|h_z(\theta_i)| dt + R_i \le \sum_{i=N}^\infty C\ell_i +C^2\ell_i^2<\infty$$
because $\sum_{i=1}^\infty \ell_i\le L<\infty$ converges. This proves the prerequisites of Theorem \ref{thm:prodconv}. Therefore $(S_k)_k$ converges uniformly on compact sets to the pp-inner function $S_\infty$. 

Therefore also $A_k=S_k^{-1}A$ converges to a contractive mvf $A_\infty=S_\infty^{-1}A$. Applying the determinant formula gives
\begin{align}
\det A_\infty(z)&=(\det S_\infty(z))^{-1} \det A(z)\\
&=c\cdot \exp\left(-\sum_{i=1}^\infty \ell_i h_z(\theta_i)+\int_0^L h_z(\theta(t))dt\right)\label{eqn:ppinnerfact5}
\end{align}
where $|c|=1$.

As in the proof of Lemma \ref{lemma:theta} we write
\begin{align}
\label{eqn:ppinnerfact6}
\int_0^L h_z(\theta(t)) dt=\int_0^{2\pi} h_z(\varphi)d\mu(\varphi)
\end{align}
where $\mu$ is the positive Borel measure corresponding to the increasing function $\theta^\dagger$. Decompose $\mu=\mu_p+\mu_c$ where $\mu_p$ is a pure point measure and $\mu_c$ an atomless measure. The pure point component $\mu_p$ is given by the jump discontinuities of $\theta^\dagger$ which in turn correspond to the intervals of constancy of $\theta$ the size of the jump being the length of the interval. That is,
\begin{align}
\label{eqn:ppinnerfact7}
\mu_p=\sum_{i=1}^\infty \ell_i \delta_{\theta_i}
\end{align}
where $\delta_\varphi$ denotes the Dirac measure at $\varphi$. Combining \eqref{eqn:ppinnerfact5}, \eqref{eqn:ppinnerfact6}, \eqref{eqn:ppinnerfact7} we get
\begin{align}
\det A_\infty(z)=c\cdot \exp\left(\int_0^{2\pi} h_z(\varphi)d\mu_c(\varphi)\right)
\end{align}
By Corollary \ref{cor:thetacont} we get
\begin{align}
A_\infty(z)=U_\infty\mint_0^{2\pi} \exp(h_z(\varphi)\,d\Sigma(\varphi))
\end{align}
for a unitary constant $U_\infty$ and a continuous increasing mvf $\Sigma$. After moving the unitary constant $U_\infty$ to the front, $A=S_\infty A_\infty$ is the desired factorization.
\end{proof}

\begin{thm}
\label{thm:ginzburg-fact}
Let $A\in\H^\infty$. Then there exist a B.P. product $B$, a pp-inner function $S_{pp}$, an sc-inner function $S_{sc}$ and an outer function $E$ such that 
\begin{align}
\label{eqn:ginzburg-fact}
A(z)=B(z)S_{pp}(z)S_{sc}(z)E(z)\,\,\,\,\text{for}\;z\in\D
\end{align}
\end{thm}

\begin{proof}
Assume without loss of generality that $A\in\Sch$ (multiply with a positive multiple of the identity matrix).
The B.P. product can be detached using Theorem \ref{thm:bpfactor}. Thus we can assume that $A$ has no zeros. Lemma \ref{lemma:ppinnerfact} reduces the claim to showing that a contractive mvf $A$ of the form
\begin{align}
\label{eqn:ginzburgpf1}
A(z)=\mint^{2\pi}_0 \exp\left(h_z(\varphi)\,d\Sigma(\varphi)\right)
\end{align}
where $\Sigma$ is a continuous increasing mvf, can be factored into an sc-inner and an outer function.

To achieve this let us decompose $\Sigma=\Sigma_s+\Sigma_a$ where $\Sigma_s$ is a singular-continuous and $\Sigma_a$ an absolutely continuous function (see e.g. \cite{SteinShakarchi3}). By continuity, we can approximate $\Sigma_s$ uniformly by a sequence of step functions $(T_k)_k$ which we may assume to be of the form
$$T_k=\sum_{i=1}^{m_k}T^{(i)}_k\ch_{(t_k^{(i-1)},t_k^{(i)}]}$$
for $m_k\in\N, 0=t_k^{(0)}<t_k^{(1)}<\cdots<t_k^{(m_k)}=2\pi$ and $T_k^{(i)}$ Hermitian matrices such that $T_k^{(1)}=0$ and $\Delta_k^{(i)}=T_k^{(i)}-T_k^{(i-1)}\ge 0$ for $m_k\ge i> 1$. Set $\Delta_k^{(1)}=0$ for convenience. Let $\Sigma^{(k)}=T_k+\Sigma_a$.

Then
$$\mint^{t_k^{(i)}}_{t_k^{(i-1)}} e^{h_z(\varphi)\,d\Sigma^{(k)}(\varphi)}
=e^{h_z(t_k^{(i-1)})\Delta_k^{(i)}}\mint^{t_k^{(i)}}_{t_k^{(i-1)}} e^{h_z(\varphi)\,d\Sigma_a(\varphi)}$$
for $i=1,\dots,m_k$. This implies
\begin{align}
\label{eqn:factpf2}
\mint_{0}^{2\pi} e^{h_z(\varphi)\,d\Sigma^{(k)}(\varphi)}
=\prod_{i=1}^{m_k}e^{h_z(t_k^{(i-1)})\Delta_k^{(i)}}\mint^{t_k^{(i)}}_{t_k^{(i-1)}} e^{h_z(\varphi)\,d\Sigma_a(\varphi)}\end{align}
We use Lemma \ref{lemma:ginzburg} to move all the factors $e^{h_z(t_k^{(i-1)})\Delta_k^{(i)}}$ up to the front. Thereby, we obtain
\begin{align}
\label{eqn:factpf3}
\mint^{2\pi}_{0} e^{h_z(\varphi)\,d\Sigma^{(k)}(\varphi)}=S_k(z)E_k(z)
\end{align}
where $S_k$, $E_k$ are contractive mvfs satisfying
\begin{align}
\label{eqn:factpf5}
\det E_k(z)=\exp\left(\int_0^{2\pi} h_z(\varphi) d\tr\Sigma_a(\varphi)\right)\;\text{and}
\end{align}
\begin{align}
\label{eqn:factpf6}
\det S_k(z)=\exp\left(\sum_{i=1}^{m_k} h_z(t_k^{(i)})\tr\Delta_k^{(i)}\right)=\exp\left(\int_0^{2\pi} h_z(\varphi) d\tr T_k(\varphi)\right)
\end{align}

By Montel's theorem, there exists a subsequence $(k_j)_j$ such that $(S_{k_j})_j$ and $(E_{k_j})_j$ both converge uniformly on compact sets to contractive mvfs $S$ and $E$, respectively. Equation \eqref{eqn:factpf5} implies
\begin{align}
\label{eqn:factpf7}
\det E(z)=\lim_{j\rightarrow\infty}\det E_{k_j}(z)=\exp\left(\int_0^{2\pi} h_z(\varphi) d\tr\Sigma_a(\varphi)\right)
\end{align}
Because the trace of an absolutely continuous function is absolutely continuous, $\det E$ is an outer function. By Lemma \ref{lemma:detinnerouter}, also $E$ must be an outer function.

We also know that $\tr T_k$ converges uniformly to $\tr\Sigma_s$ as $k\rightarrow\infty$. Therefore \eqref{eqn:factpf6} gives
\begin{align}
\label{eqn:factpf8}
\det S(z)=\lim_{j\rightarrow\infty}\det S_{k_j}(z)=\exp\left(\int_0^{2\pi} h_z(\varphi) d\tr \Sigma_s(\varphi)\right)
\end{align}
The trace of a singular continuous mvf is singular continuous. Hence Lemma \ref{lemma:detinnerouter} shows that $S$ is an sc-inner mvf.

By Helly's convergence Theorem \ref{thm:hellyconvmvf}, the left hand side of \eqref{eqn:factpf3} converges to $A(z)$ for all $z\in\D$. Therefore we obtain the desired factorization:
\begin{align}
A(z)=S(z)E(z)
\end{align}
\end{proof}

We proceed proving some additional properties of inner and outer functions.

\begin{lemma}
\label{lemma:innerouterconst}
A function $A\in\H^\infty$ that is both inner and outer must be a unitary constant.
\end{lemma}

\begin{proof}
If $A$ is inner and outer, then also $\det A$ is a scalar function which is inner and outer. Hence $\det A(z)$ is a unimodular constant. By Potapov's Theorem \ref{thm:potapov}, $A$ is given by \eqref{eqn:potapovrepr}. Plugging in the value $z=0$ gives
$$1=|\det A(0)|=\exp(-L)$$ 
and therefore $L=0$.
\end{proof}

\begin{thm}
\label{thm:innermintrep}
Let $A\in\H^\infty$. Then $A$ is an inner function if and only if there exist a B.P. product $B$, a pp-inner function $S_{pp}$ and an sc-inner function $S_{sc}$ such that
\begin{align}
\label{eqn:innerfct1}
A(z)=B(z)S_{pp}(z)S_{sc}(z)
\end{align}
for all $z\in\D$.
\end{thm}

\begin{proof}We already saw that functions of the form \eqref{eqn:innerfct1} are inner. The inner-outer factorization \eqref{eqn:ginzburg-fact} and Lemma \ref{lemma:innerouterconst} imply the other direction.
\end{proof}

\begin{thm}
\label{thm:outermintrep}
A function $A\in\H^\infty$ is outer if and only if 
for every $B\in\H^\infty$ with the property $B^*(e^{i\theta})B(e^{i\theta})=A^*(e^{i\theta})A(e^{i\theta})$ for almost every $\theta\in[0,2\pi]$, we have that
\begin{align}
\label{eqn:outerdef2}
B^*(z)B(z)\le A^*(z)A(z)
\end{align}
for all $z\in\D$.
\end{thm}
Ginzburg \cite{Ginzburg} uses this as the definition of outerness. The condition can be shown to be equivalent to a Beurling-type definition\index{Beurling} of outer functions by invariant subspaces of a certain shift operator (see \cite{Sz.-Nagy}).
\begin{proof}
Suppose that $A$ is an outer function and let $B\in\H^\infty$ be such that $A^*A=B^*B$ holds a.e. on $\T$.\\
By Theorems \ref{thm:ginzburg-fact} and \ref{thm:innermintrep} we have $B=X\cdot E$ for an inner function $X$ and an outer function $E$. Then $A^*A=E^* X^* X E=E^*E$ a.e. on $\T$. Therefore, $Y=A\cdot E^{-1}$ is an inner function. But it is also outer. So $A(z)=U\cdot E(z)$ for a unitary constant $U$. Since $X$ is an inner function, $X^*X\le I$ on $\D$. Hence
$$B^*B=E^* X^* X E \le E^* E = A^* A.$$
\end{proof}
Note that the proof also shows that outer functions are, up to a unitary constant, uniquely determined by their radial limit functions.
 
\begin{thm}
\label{thm:detinnerouter}
A mvf $A\in\H^\infty$ is inner (resp. outer) if and only if $\det A$ is a scalar inner (resp. outer) function.
\end{thm}

\begin{proof}
We already showed the part for outer functions in Lemma \ref{lemma:detinnerouter}. The part for inner functions follows similarly using Theorem \ref{thm:innermintrep}.
\end{proof}

We still have to show that the factorization \eqref{eqn:ginzburg-fact} obtained in Theorem \ref{thm:ginzburg-fact} is unique.
\begin{thm}
\label{thm:fact-unique}
The functions $B, S_{pp}, S_{sc}, E$ in Theorem \ref{thm:ginzburg-fact} are uniquely determined up to multiplication with a unitary constant.
\end{thm}

\begin{proof}
Uniqueness of the B.P. product was shown already in Theorem \ref{thm:bpfactor}. Suppose that
$$S_{pp}S_{sc}E=\tilde{S}_{pp}\tilde{S}_{sc}\tilde{E}$$
are two factorizations into pp-inner, sc-inner and outer factors. Then $$\tilde{S}_{sc}^{-1}\tilde{S}_{pp}^{-1}S_{pp}S_{sc}=\tilde{E}\cdot E^{-1}$$
The left hand side is an inner function and the right hand side an outer function. By Lemma \ref{lemma:innerouterconst}, $\tilde{E}(z)=U\cdot E(z)$ for some unitary constant $U$. Now we have $\tilde{S}_{pp}^{-1}S_{pp}=\tilde{S}_{sc}U S_{sc}^{-1}$. The left hand side is a pp-inner function and the right hand side an sc-inner function. This holds also for their determinants. By the scalar uniqueness, both sides must be unitary constants.
\end{proof}

\noindent Theorems \ref{thm:ginzburg-fact} and \ref{thm:fact-unique} prove the first part of Theorem \ref{thm:main}.

\subsection{Uniqueness}

Now we proceed to proving the remaining claim of Theorem \ref{thm:main}.

\begin{thm}Let $A_i\in\H^\infty$, $i=1,2$ with
$$A_i(z)=\mint^{2\pi}_0 \exp(h_z(\varphi)M_i(\varphi)d\varphi)$$
for $z\in\D$ and $M_i$ Lebesgue integrable Hermitian mvfs. Assume $A_1\equiv A_2$ on $\D$. Then $M_1=M_2$ almost everywhere on $[0,2\pi]$. That is, the function $M$ in the representation \eqref{eqn:outer} of an outer function is uniquely determined up to changes on a set of measure zero.
\end{thm}

\begin{proof}
Set
$$A_i(z,t)=\mint_0^t \exp(h_z(\varphi)M_i(\varphi)d\varphi)\,\,\text{and}\,\,
\tilde{A}_i(z,t)=\mint_t^{2\pi} \exp(h_z(\varphi)M_i(\varphi)d\varphi)
$$
for $z\in\D$, $t\in[0,2\pi]$, $i=1,2$. These are outer functions for every fixed $t$. Then
$$A_1(z,t)\tilde{A}_1(z,t)=A_1(z,t)=A_2(z,t)=A_2(z,t)\tilde{A}_2(z,t)$$
and consequently
\begin{align}
\label{eqn:uniquenesspf}
X(z,t)=A_2(z,t)^{-1}A_1(z,t)=\tilde{A}_2(z,t)\tilde{A}_1(z,t)^{-1}
\end{align}
for all $z\in\D$, $t\in[0,2\pi]$. For every $t$, $X(\cdot,t)$ is an outer function. But for every $\theta\in(t,2\pi)$, the radial limit $A_i(e^{i\theta},t)$ is unitary. Consequently, $X(e^{i\theta},t)$ is unitary for $t\in (t,2\pi)$. Likewise, $\tilde{A}_i(e^{i\theta},t)$ is unitary for all $\theta\in(0,t)$. Together, \eqref{eqn:uniquenesspf} implies that $X(e^{i\theta},t)$ is unitary for almost all $\theta\in[0,2\pi]$, i.e. $X(\cdot,t)$ is an inner function. Therefore, $X(\cdot,t)$ must be a unitary constant. That is, there exist unitary matrices $U(t)$ for $t\in[0,2\pi]$ such that
\begin{align}
\label{eqn:uniquenesspf2}
A_1(z,t)=A_2(z,t)U(t)
\end{align}
for all $z\in\D$. Since $A_i(0,t)=\mint_0^{t}\exp(-M_i(\varphi)d\varphi)$ are  positive matrices, $U(0)$ is also positive, i.e. $U(0)=I$. So we conclude
\begin{align}
\mint_0^{t} \exp(-M_1(\varphi)d\varphi)=\mint_0^{t} \exp(-M_2(\varphi)d\varphi)
\end{align}
Deriving this equation with respect to $t$ gives
\begin{align}
-M_1(t)A_1(0,t)=-M_2(t)A_2(0,t)
\end{align}
for almost all $t\in[0,2\pi]$. Since $A_1(0,t)=A_2(0,t)$, the assertion follows.
\end{proof}

It is natural to ask whether the $E_k$ and $S$ in the representations of pp-inner and sc-inner functions, respectively, are uniquely determined. The question for uniqueness of $S$ in the representation of an sc-inner function remains unresolved for now. For the pp-inner part, the answer is negative. To see that there is no uniqueness for the $E_k$ let us consider the following. 
\begin{example}
Let 
$$E_1(t)=\begin{pmatrix}t^2/2 & 0\\0 & t-t^2/2\end{pmatrix}\;\text{and}\; E_2(t)=\begin{pmatrix}t-t^2/2 & 0\\ 0 & t^2/2\end{pmatrix}$$
for $t\in[0,1]$. The mvfs $E_1$, $E_2$ are increasing and satisfy $\tr E_1(t)=\tr E_2(t)=t$. Then
$$\mint_0^1 \exp\left(\frac{z+1}{z-1}\, dE_i(t)\right)=\begin{pmatrix}e^{\frac{z+1}{2(z-1)}} & 0\\
0 & e^{\frac{z+1}{2(z-1)}}\end{pmatrix}$$
for $i=1,2$, but $E_1(t)\not=E_2(t)$ for almost all $t\in[0,1]$.
\end{example}

Despite the apparent non-uniqueness, it turns out that sometimes $E_k$ \emph{is} uniquely determined and in fact there exists a criterion to decide when that is the case \cite[Theorem 0.2]{Arov-DymII}, \cite[Theorem 30.1]{Brodskii}. The proof of this however relies on an elaborate operator theoretic framework, which would lead us too far astray here.

\appendix
\section{Background}
\label{sect:background}

In this section we collect some basic facts on matrix norms and analytic matrix-valued functions.

\subsection{Properties of the matrix norm}
\label{subsect:matrixnorm}
The matrix norm\index{matrix norm} we use throughout this work is the operator norm\index{operator norm} given by
$$\|A\|=\sup_{\|v\|=1}\|Av\|$$
for $A\in M_n$, where $\|v\|$ denotes the Euclidean norm of $v\in\C^n$. For reasons which are apparent from Lemma \ref{lemma:matrixnorm},      this norm is often referred to as spectral norm\index{spectral norm}. \begin{lemma}
\label{lemma:matrixnorm}
Let $U\in M_n$ be unitary, $A,B\in M_n$ arbitrary matrices, $D=\diag(\lambda_1,\dots,\lambda_n)$ a diagonal matrix and $v\in\C^n$. Then
\begin{enumerate}
\item[(1)] $\|AB\|\le\|A\|\cdot\|B\|$ \tinypf{$\|ABv\|\le \|A\| \|Bv\|$ holds for all $\|v\|=1$.}
\item[(2)] $\|A\|=\|A^*\|$ \tinypf{$\|Av\|^2=\langle Av,Av\rangle=\langle v,A\adj Av\rangle\le\|v\|\cdot\|A\adj\|\cdot\|Av\|$ $\Rightarrow$ $\|A\|\le\|A\adj\|$. '$\ge$' follows by symmetry.}
\item[(3)] $\|AU\|=\|UA\|=\|A\|$ \tinypf{(a) $\|UAv\|^2=\langle UAv,UAv\rangle=\langle Av,Av\rangle=\|Av\|^2$ $\Rightarrow$ $\|UA\|=\|A\|$. (b) $\|AU\|=\sup_{v\not=0} \frac{\|AUv\|}{\|v\|}=\sup_{w\not=0} \frac{\|Aw\|}{\|U\adj w\|}=\|A\|$}
\item[(4)] $\|D\|=\max_{i}|\lambda_i|$
\tinypf{'$\ge$': $\|D\|\ge \|De_i\|=\|\lambda_i\|$. '$\le$': $v=(x_1,\dots,x_n)^\perp$ $\Rightarrow$ $\|Dv\|^2=\sum_i |\lambda_i x_i|^2\le \max_i |\lambda_i|^2 \cdot \|v\|^2$}
\item[(5)] $A$ Hermitian $\Rightarrow$ $\|A\|=\sigmamax(A)$
\tinypf{$A$ is unitarily diagonalizable $\Rightarrow$ $A=UDU\adj$ $\Rightarrow$
$\|A\|\overset{(3)}{=}\|D\|\overset{(4)}{=}\sigmamax(A)$}
\item[(6)] $\|A\|=\sqrt{\sigmamax(AA^*)}=\sqrt{\|AA^*\|}$
\tinypf{$\|AA\adj\|=\sigmamax(AA\adj)$ by (5). First $\|AA\adj\|\le\|A\|\cdot\|A\adj\|=\|A\|^2$. For $\|v\|=1$ we have
$\|A\adj v\|^2=\langle A\adj v,A\adj v\rangle=\langle v,AA\adj v\rangle\le\|AA\adj v\|\le\|AA\adj\|$, so $\|A\|^2=\|A\adj\|^2\le\|AA\adj\|$.}
\item[(7)] $\|U\|=1$
\tinypf{$\|U\|\overset{(6)}{=}\sqrt{\|UU\adj\|}=1$}
\item[(8)] $A\ge 0$ $\Rightarrow$ $\|A\|\le\tr A$
\tinypf{$\|A\|=\sigmamax(A)\le\tr(A)$}
\item[(9)] $\|A\|=\sup_{\|x\|=1}\sqrt{|x^* AA^*x|}$
\tinypf{By (6) it suffices to show $\|A\|=\sup_{\|x\|=1}|x^*Ax|$ for $A$ Hermitian. For $\|x\|=1$ note $x^*Ax=\sum_{i,j}A_{ij}\overline{x}_i x_j=\sum_{i} (Ax)_i \overline{x}_i=\langle Ax,x\rangle$. So $|x^*Ax|\le \|A\|$. Taking the supremum gives '$\ge$' of the claim. The other direction follows from (5) by plugging in an eigenvector corresponding to the greatest eigenvalue of $A$.}
\item[(10)] $\|A\|=\sup_{\|x\|=\|y\|=1}|x^* Ay|$
\tinypf{Same argument as in (9).}
\item[(11)] $\|A\|\ge \sqrt{\sum_j |A_{ij}|^2}$ for all $i=1,\dots,n$.
\tinypf{Plug in $x=e_i$ in (9).}
\item[(12)] $\|A\|\ge |A_{ij}|$ for all $i,j=1,\dots,n$.
\tinypf{Plug in $x=e_i$, $y=e_j$ in (10).}
\item[(13)] $\|A^{-1}\|\le \frac{\|A\|^{n-1}}{|\det A|}$ if $\det A\not=0$
\tinypf{By homogeneity we may assume $\|A\|=1$. Let $\lambda$ be the smallest eigenvalue of $AA\adj$. Then $\lambda\ge \det(AA^*)=(\det A)^2$ by (6). Note that $\lambda^{-1}$ is the greatest eigenvalue of $(AA\adj)^{-1}$, so $\lambda^{-1}=\|(A^{-1})(A^{-1})\adj\|=\|A^{-1}\|^2$. Together $\|A^{-1}\|=\lambda^{-1/2}\le \frac{1}{|\det A|}$.}
\end{enumerate}
\end{lemma}
\begin{proof}
\begin{enumerate}
\item[(1)] $\|ABv\|\le \|A\| \|Bv\|$ holds for all $\|v\|=1$.
\item[(2)] $\|Av\|^2=\langle Av,Av\rangle=\langle v,A\adj Av\rangle\le\|v\|\cdot\|A\adj\|\cdot\|Av\|$ $\Rightarrow$ $\|A\|\le\|A\adj\|$. '$\ge$' follows by symmetry.
\item[(3)] (a) $\|UAv\|^2=\langle UAv,UAv\rangle=\langle Av,Av\rangle=\|Av\|^2$ $\Rightarrow$ $\|UA\|=\|A\|$. (b) $\|AU\|=\sup_{v\not=0} \frac{\|AUv\|}{\|v\|}=\sup_{w\not=0} \frac{\|Aw\|}{\|U\adj w\|}=\|A\|$
\item[(4)] '$\ge$': $\|D\|\ge \|De_i\|=\|\lambda_i\|$. '$\le$': $v=(x_1,\dots,x_n)^\perp$ $\Rightarrow$ $\|Dv\|^2=\sum_i |\lambda_i x_i|^2\le \max_i |\lambda_i|^2 \cdot \|v\|^2$
\item[(5)] $A$ is unitarily diagonalizable $\Rightarrow$ $A=UDU\adj$ $\Rightarrow$
$\|A\|\overset{(3)}{=}\|D\|\overset{(4)}{=}\sigmamax(A)$
\item[(6)] $\|AA\adj\|=\sigmamax(AA\adj)$ by (5). First $\|AA\adj\|\le\|A\|\cdot\|A\adj\|=\|A\|^2$. For $\|v\|=1$ we have
$\|A\adj v\|^2=\langle A\adj v,A\adj v\rangle=\langle v,AA\adj v\rangle\le\|AA\adj v\|\le\|AA\adj\|$, so $\|A\|^2=\|A\adj\|^2\le\|AA\adj\|$.
\item[(7)] $\|U\|\overset{(6)}{=}\sqrt{\|UU\adj\|}=1$
\item[(8)] $\|A\|=\sigmamax(A)\le\tr(A)$
\item[(9)] By (6) it suffices to show $\|A\|=\sup_{\|x\|=1}|x^*Ax|$ for $A$ Hermitian. Let $\|x\|=1$. Note $x^*Ax=\sum_{i,j}A_{ij}\overline{x}_i x_j=\sum_{i} (Ax)_i \overline{x}_i=\langle x,Ax\rangle$. So $|x^*Ax|\le \|A\|$. Taking the supremum gives '$\ge$' of the claim. The other direction follows from (5) by plugging in an eigenvector corresponding to the greatest eigenvalue of $A$.
\item[(10)] Same argument as in (9).
\item[(11)] Plug in $x=e_i$ in (9).
\item[(12)] Plug in $x=e_i$, $y=e_j$ in (10).
\item[(13)] By homogeneity we may assume $\|A\|=1$. Let $\lambda$ be the smallest eigenvalue of $AA\adj$. Then $\lambda\ge \det(AA^*)=(\det A)^2$ by (6). Note that $\lambda^{-1}$ is the greatest eigenvalue of $(AA\adj)^{-1}$, so $\lambda^{-1}=\|(A^{-1})(A^{-1})\adj\|=\|A^{-1}\|^2$. Together $\|A^{-1}\|=\lambda^{-1/2}\le \frac{1}{|\det A|}$.
\end{enumerate}

\end{proof}

The eigenvalues of the positive matrix $AA^*$ are also called the \defin{singular values}\index{singular value} of $A$. So property (6) says that the matrix norm of $A$ is the square root of the largest singular value.\\

The following fact is often useful when dealing with contractive mvfs and makes the particular choice of the spectral norm especially convenient.
\begin{lemma}
\label{lemma:normcontr}
A matrix $A\in M_n$ satisfies $\|A\|\le 1$ if and only if $I-AA^*\ge 0$.
\end{lemma}

\begin{proof}
Let $AA^*=UDU^*$ with $U$ a unitary and $D$ a diagonal matrix. Then $I-AA^*\ge 0$ is equivalent to $I-D\ge 0$ and this is equivalent to $\|A\|\le 1$ since the diagonal entries of $D$ are the singular values of $A$.
\end{proof}

\subsection{Analytic matrix-valued functions}
\label{analyticmatrixfcts}
By an \defin{analytic matrix-valued function (analytic mvf)}\index{analytic mvf} in the unit disk $\D\subset\C$ we mean a function $A:\D\rightarrow M_n$ all components of which are holomorphic throughout $\D$. The terms \defin{rational}, \defin{continuous}, \defin{differentiable}, \defin{meromorphic} and \defin{harmonic} are defined the same way for mvfs. It is clear that the basic implements of ``non-multiplicative" complex analysis, i.e. Laurent development, Cauchy and Poisson integral formula, mean value property etc., immediately carry over to the matrix-valued case.
An analytic mvf $A$ is called \defin{bounded} if
$$\|A\|_\infty=\sup_{z\in\D}\|A(z)\|<\infty$$
The space $\H^\infty$ is defined as the set of bounded analytic mvfs on the unit disk, whose determinant does not vanish identically.

\begin{lemma}
\label{unifconvlemma}
Let $(A_k)_k$ be a sequence of analytic mvfs on $\D$ which converges uniformly on compact subsets to a mvf $A$. Then $A$ is analytic.
\end{lemma}

\begin{proof}This follows just like in the scalar case using the appropriate analogoues of Cauchy's and Morera's theorems (cf. \cite[Theorem 10.28]{Rudin2}).
\end{proof}

\begin{lemma}\index{subharmonic}
\label{subharmlemma}
Let $A$ be an analytic mvf on $\D$. Then $\|A\|$ is subharmonic.
\end{lemma}

\begin{proof}
Let $z_0\in\D$ and $\gamma(t)=z_0+re^{it}$ for $t\in[0,2\pi)$ where $0\le r<1$ is such that $\gamma(t)\in\D$ for all $t\in[0,2\pi)$. Then by the Cauchy integral formula applied to the components of $A$ we get
$$A(z_0)=\frac{1}{2\pi i}\int_\gamma \frac{A(z)}{z-z_0} dz
=\frac{1}{2\pi}\int_0^{2\pi} A(\gamma(t))\,dt$$
Hence the triangle inequality gives $\|A(z_0)\|\le\frac{1}{2\pi}\int_0^{2\pi}\|A(z_0+re^{it})\|\,dt$ which proves the claim.
\end{proof}

By saying that $A:\D\rightarrow M_n$ is a \emph{contractive} mvf\index{contractive mvf}, we mean that $\|A(z)\|\le 1$ for all $z\in\D$. The space of all contractive analytic mvfs, whose determinant does not vanish identically is denoted by $\Sch\subset\H^\infty$.\index{Schur class}
The following rank invariance statement is from \cite[Chapter I]{Potapov}.

\begin{lemma}\index{rank invariance lemma}
\label{rankinvar}
Let $A\in\Sch$ be such that 
$I-A(z_0)^*A(z_0)$ has rank $0\le r<n$ for some $z_0\in\D$. Then there exist unitary matrices $U,V$ and a contractive analytic mvf $\tilde{A}:\D\rightarrow M_r$ such that
$$A(z)=U\left(\begin{array}{cc}\tilde{A}(z) & 0\\0 & I_{n-r}\end{array}\right)V$$ 
for all $z\in\D$. In particular, $I-A(z)A^*(z)$ has rank $r$ throughout $\D$.
\end{lemma}

\begin{proof}
By assumption, $1$ is an eigenvalue of $A(z_0)A^*(z_0)$ and the dimension of the corresponding eigenspace is $n-r$. Hence, by singular value decomposition, we can find unitary matrices $U, V$ and a $r\times r$ diagonal matrix $D$ such that
$$A(z_0)=U\left(\begin{array}{cc}D & 0\\0 & I_{n-r}\end{array}\right)V$$
Now consider the contractive analytic matrix function $B(z)=U^* A(z) V^*$. Since $|B_{ij}(z)|\le\|B(z)\|=\|A(z)\|\le 1$, the entries of $B$ are analytic functions mapping $\D$ to $\overline{\D}$. We have by construction that $B_{ii}(z_0)=1$ for $i=r+1,\dots,n$ which implies by the maximum principle that $B_{ii}(z)=1$ for all $z\in\D$, $i=r+1,\dots,n$. Finally, property (11) from above implies that the off-diagonal entries in the rows and columns $r+1,\dots,n$ of $B$ are all constant $0$. Summing up, $B$ is of the form 
$$B(z)=\left(\begin{array}{cc}\tilde{A}(z) & 0\\0 & I_{n-r}\end{array}\right)$$
for some contractive analytic $r\times r$ matrix-valued function $\tilde{A}$.
\end{proof}

We can use this to obtain the following analogue of the strong maximum principle for holomorphic functions on the unit disk.

\begin{cor}\label{cor:unitaryconst}\index{maximum principle}
A contractive analytic mvf, which is unitary at some $z\in\D$, is constant.\end{cor}

\begin{proof}This is the case $r=0$ in the previous lemma.\end{proof}

An important tool for extracting convergent subsequences from contractive analytic mvfs is Montel's theorem.

\begin{thm}[Montel]\index{Montel's theorem}
\label{thm:montel}
Let $(A_k)_k$ be a sequence of contractive analytic mvfs. Then there exists a subsequence $(A_{k_j})_j$ which converges uniformly on compact sets to a contractive analytic mvf.
\end{thm}

\begin{proof}
This is a consequence of the theorem of Arzel\`a-Ascoli. Equicontinuity of the sequence follows, because it is uniformly bounded and also the sequence of its derivatives is uniformly bounded by the Cauchy integral formula. 
\end{proof}

\section{The Herglotz representation theorem for mvfs}
\label{sect:herglotz}

The aim of this section is to prove a generalization of the Herglotz representation theorem for positive harmonic functions (see \cite[Theorem I.3.5 (c)]{Garnett}) to the matrix-valued case. The proof uses Helly's classical (scalar) theorems, which we will prove first.
\begin{thm}[Herglotz]
\label{thm:herglotz}\index{Herglotz representation theorem}
Let $A:\D\rightarrow M_n$ be holomorphic and $\Im A(z)=\frac{1}{2i}(A(z)-A^*(z))\ge 0$ for $z\in\D$. Then there exists a Hermitian matrix $A_0$ and an increasing mvf $\sigma:[0,2\pi]\rightarrow M_n$ such that
$$A(z)=A_0+i\int^{2\pi}_0 \frac{e^{it}+z}{e^{it}-z}\,d\sigma(t)\hspace{1cm}(z\in\D)$$
\end{thm}
The integral used in this theorem is the classical Riemann-Stieltjes integral with respect to a matrix-valued function. Integrals of this type are discussed for instance in \cite{Rudin1}, \cite{Apostol} and in a more general setting in \cite[\S 4]{Brodskii}.

\subsection{Helly's classical theorems}

The next proposition is also known as Helly's first theorem and provides a compactness result for BV-functions with respect to pointwise convergence.

\begin{thm}[Helly's selection theorem, scalar version]
\label{thm:hellyselsc}\index{Helly's selection theorem}\index{Helly's first theorem}
Let $(f_k)_k$ be a sequence of functions in $\BV([a,b])$ with uniformly bounded total variation. Then there exists a subsequence $(f_{k_j})_j$ such that $f_{k_j}$ converges pointwise to some function $f\in\BV([a,b])$. 
\end{thm}
\begin{proof}
Let us first prove the statement in the case that $(f_k)_k$ is a uniformly bounded sequence of increasing functions.
By the usual diagonal subsequence argument we can use the uniform boundedness to extract a subsequence $(f_{k_j})_j$ which converges at all rational points in $[a,b]$ to an increasing function $f$ on $[a,b]\cap\Q$. We extend $f$ to all of $[a,b]$ by
$$f(t)=\inf\{f(q)\,:\,q\in[t,b]\cap\Q\}\hspace{1cm}(t\in[a,b])$$
Now $f$ is by construction increasing on $[a,b]$. It remains to discuss the convergence of $f_{k_j}(t)$ for $t\in[a,b]\cap\Q^c$. If $f$ is continuous at $t$, the density of $\Q\cap[a,b]$ in $[a,b]$ implies $f_{k_j}(t)\rightarrow f(t)$. Since $f$ is only discontinuous at countably many points, we may achieve convergence everywhere by repeating the diagonal subsequence argument on the set of discontinuities, therefore choosing a further subsequence. This finishes the proof for increasing functions.

Now assume that $(f_k)_k$ is a sequence of BV-functions with uniformly bounded total variation. We decompose $f_k$ into increasing functions by writing
$$f_k=g_k-h_k$$ 
with $g_k(t)=\var_{[a,t]} f_k$ and $h_k=g_k-f_k$. Assume $|f_k(t)|\le\var_{[a,b]} f_k\le C$ for all $k$. Then also $|g_k(t)|\le C$ and $|h_k(t)|\le 2C$ for all $k$ and $t\in[a,b]$. Thus we can use the statement for increasing functions to choose a subsequence such that both $(g_{k_j})_j$ and $(h_{k_j})_j$ converge pointwise to increasing functions $g$ and $h$, respectively. That implies $f_k\rightarrow f=g-h$ pointwise. Finally, $f$ is also of bounded variation since for every fixed partition $\tau$ of $[a,b]$ we have that $\lim_{k\rightarrow\infty} \var_{[a,b]}^\tau f_k=\var_{[a,b]}^\tau f$. Therefore $\var_{[a,b]} f\le C$.
\end{proof}

The next theorem is a convergence theorem for Riemann-Stieltjes integrals. It is also known as Helly's second theorem.

\begin{thm}[Helly's convergence theorem, scalar version]
\label{thm:hellyconvsc}\index{Helly's convergence theorem}\index{Helly's second theorem}
Let $(f_n)_n$ be a sequence in $\BV([a,b])$ with uniformly bounded total variation which converges pointwise to some function $f$ on $[a,b]$. Then $f\in\BV([a,b])$ and for every $\varphi\in C([a,b])$ we have
\begin{align*}
\lim_{k\rightarrow\infty} \int^b_a \varphi(t)\,df_k(t)=\int^b_a \varphi(t)\,df(t)
\end{align*}
\end{thm}

\begin{proof}
Choose $C>0$ such that $\var_{[a,b]} f_k\le C$ for all $k$.
First note that $f\in\BV([a,b])$, since for any partition $\tau$ of $[a,b]$, $\var_{[a,b]}^\tau f=\lim_{k\rightarrow\infty} \var_{[a,b]}^\tau f_k$ and therefore also $\var_{[a,b]} f\le C$.

Let $\epsilon>0$. Since $\varphi$ is uniformly continuous on $[a,b]$, there exists $\delta>0$ such that $|\varphi(t)-\varphi(s)|<\frac{\epsilon}{2C}$ for $|t-s|<\delta$.

Now whenever $(\tau,\xi)\in\mathcal{T}_a^b$ is a tagged partition such that $\nu(\tau)<\delta$, we get

\begin{align*}
\left|\int_a^b \varphi df-\sum_{i=1}^m \varphi(\xi_i)\Delta_i f\right| &=\left|\sum_{i=1}^m\int_{t_{i-1}}^{t_i} (\varphi(t)-\varphi(\xi_i))df(t)\right|\le \left(\frac{\epsilon}{2C}\right) \sum_{i=1}^m \var_{[t_{i-1},t_i]} f\\
&=\frac{\epsilon}{2C}\cdot \var_{[a,b]} f\le \frac{\epsilon}{2}
\end{align*}

The same calculation gives also

$$\left|\int_a^b \varphi df_k - \sum_{i=1}^m \varphi(\xi_i)\Delta_i f_k\right|\le \frac{\epsilon}{2}$$
for all $k$. Therefore

\begin{align*}
\left|\int^b_a \varphi\,df_k - \int^b_a\varphi\,df\right|
&\le \left|\int^b_a \varphi\, df_k - \sum^m_{i=1}\varphi(\xi_i)\Delta_i f_k\right|+\left|\sum^m_{i=1}\varphi(\xi_i)(\Delta_i f_k-\Delta_i f)\right|\\
&+\left|\sum^m_{i=1}\varphi(\xi_i)\Delta_i f-\int^b_a \varphi\,df\right|
\le \epsilon + \|\varphi\|_\infty \sum^m_{i=1} |\Delta_i f_k-\Delta_i f|
\end{align*}

Since $f_k\rightarrow f$ pointwise, we can let $k\rightarrow\infty$, while letting the tagged partition $(\tau,\xi)$ fixed, and obtain

$$\limsup_{k\rightarrow\infty} \left|\int^b_a \varphi\,df_k - \int^b_a\varphi\,df\right|\le \epsilon$$
Since $\epsilon$ was arbitrary, the proposition is proven.
\end{proof}
One can use Helly's theorems to prove the familiar characterization of continuous functionals on $C([a,b])$ as being given by Stieltjes-integrating against $\BV$-functions. This is the classical Riesz representation theorem.\index{Riesz representation theorem} In view of that, Helly's theorems also provide a more explicit insight in the weak-* compactness of the unit ball in $C([a,b])^\prime$ than is offered by the fairly abstract proof of the Banach-Alaoglu theorem.\index{Banach-Alaoglu theorem}

\subsection{Proof of the theorem}
In this section we prove the Herglotz representation theorem. The proof is based on the proof in the scalar case, just that we use Helly's theorems instead of Riesz' representation theorem and Banach-Alaoglu. This has the benefit that we are simply dealing with pointwise converging sequences of functions.

\emph{Proof of Theorem \ref{thm:herglotz}.} 
We assume $A(0)=0$. Set $T(z)=\Im A(z)$ for $z\in\D$ and $T^{(r)}(t)=T(re^{it})$ for $t\in[0,2\pi]$ and $0<r<1$. Then the scalar mean value property for harmonic functions gives
\begin{align*}
\frac{1}{2\pi}\int^{2\pi}_0 \|T(re^{it})\|\,dt\le \frac{1}{2\pi}\int^{2\pi}_0 \tr T(re^{it})\,dt=\tr T(0)=C
\end{align*}
Hence also $\|T^{(r)}_{ij}\|_1\le C$ for all $1\le i,j\le n$ where $\|\cdot\|_1$ denotes the $1$-norm w.r.t. the Lebesgue measure on $[0,2\pi]$. Now define
$$\sigma^{(r)}_{ij}(t)=\frac{1}{2\pi}\int_0^{t}T_{ij}^{(r)}(s)\,ds$$
Note that $\sigma_{ij}^{(r)}$ are BV-functions with $\var_{[0,2\pi]} \sigma_{ij}^{(r)}\le C$ and the mvfs $\sigma^{(r)}(t)=(\sigma_{ij}^{(r)}(t))_{ij}$ are increasing, because $T^{r}(t)\ge 0$. Now we can apply Helly's selection Theorem \ref{thm:hellyselsc} to obtain a sequence $(r_k)_k$ with $r_k\rightarrow 1-$ and a $\BV$-function $\sigma_{ij}$ such that $$\sigma_{ij}^{(r_k)}\longrightarrow \sigma_{ij}$$
pointwise on $[0,2\pi]$. By taking appropriate subsequences we can assume without loss of generality that this holds for all $(i,j)$.

The mvf $\sigma=(\sigma_{ij})_{ij}$ is increasing, as it is the pointwise limit of increasing mvfs. The functions $A^{(r)}(z)=A(rz),\,z\in\D$ are holomorphic on $\D$ and continuous on $\overline{\D}$. Hence the Poisson integral formula for holomorphic functions implies
\begin{align*}
A^{(r)}(z)=\frac{i}{2\pi}\int_0^{2\pi}\frac{e^{it}+z}{e^{it}-z}T^{(r)}(t)\,dt
\end{align*}
Let $z\in\D$ be fixed. Applying the above to $f(t)=\frac{e^{it}+z}{e^{it}-z}$ we get
\begin{align*}
A(z) &= \lim_{k\rightarrow\infty} A^{(r_k)}(z) = \lim_{k\rightarrow\infty} i\int_0^{2\pi} \frac{e^{it}+z}{e^{it}-z}\,d\sigma^{(r_k)}(t) = i\int_0^{2\pi} \frac{e^{it}+z}{e^{it}-z}\,d\sigma(t)
\end{align*}
where the last equality follows from Helly's convergence Theorem \ref{thm:hellyconvsc}.
\qed

\section*{Notation}
\addcontentsline{toc}{section}{Notation}

\begin{tabular}{ll}
$\D$	& unit disk around the origin in $\C$\\
$\T$	& unit circle around the origin in $\C$\\
$\Hpl$ & upper half plane in $\C$\\
$M_n$ & space of $n\times n$ matrices with entries in $\C$\\
$I=I_n$   & unit matrix in $M_n$\\
$A^*$ & conjugate transpose of $A\in M_n$\\
$A_{ij}$ & $(i,j)$th entry of the matrix $A$\\
$A\ge 0$ & the Hermitian matrix $A$ is positive semidefinite\\
$A>0$ & the Hermitian matrix $A$ is positive definite\\
$\|A\|$ & operator norm of the matrix $A$\\
$\tr A$ & trace of the matrix $A$\\
$\sigma(A)$ & set of eigenvalues of $A$\\
$\sigmamax(A)$ & spectral radius of $A$, i.e. largest eigenvalue\\
$\H^\infty$ & bounded analytic matrix functions on $\D$ with $\det A\not\equiv0$\\
$\Sch\subset\H^\infty$ & subspace of contractive functions\\
$\prodr=\prod$ & product of matrices ordered from left to right\\
$\prodl$ & product of matrices ordered from right to left\\
$\mint$ & (left-)multiplicative integral\\
$\Re A$ & real part of the matrix $A$, given by $\frac{1}{2}(A+A^*)$\\
$\Im A$ & imaginary part of the matrix $A$, given by $\frac{1}{2i}(A-A^*)$\\
$C(K)$ & space of continuous function on a compact set $K\subset\R^n$\\
$\BV([a,b];M_n)$ & space of matrix-valued bounded variation functions on $[a,b]\subset\R$\\
$\BV([a,b])$ & short for $\BV([a,b];M_1)$\\
$\var_{[a,b]} f$ & variation of a matrix-valued function $f$ on the interval $[a,b]$\\
$\osc_{[a,b]} f$ & oscillation of a function $f$ on the interval $[a,b]$\\
$f\in\mathcal{M}^b_a[E]$ & the multiplicative integral $\mint^b_a \exp(f\,dE)$ exists\\
$h_z(\theta)$ & the Herglotz kernel $h_z(\theta)=\frac{z+e^{i\theta}}{z-e^{i\theta}}$\\
$\ch_M$ & characteristic function of the set $M$\\
$\theta^\dagger$ & generalized inverse of the increasing function $\theta$\\
$\lambda$ & Lebesgue measure on $\R$
\end{tabular}

\newpage
\addcontentsline{toc}{section}{References}
\bibliographystyle{plain}
\bibliography{References}

\newpage
\addcontentsline{toc}{section}{Index}
\printindex
\end{document}